\documentclass[10pt, reqno]{amsart}

\usepackage{amsthm,amsfonts,amsmath,amssymb,mathrsfs,enumitem,comment,mathtools,euler}
\usepackage[mathscr]{eucal}
\usepackage[utf8]{inputenc}
\usepackage[T1]{fontenc}
\usepackage[colorlinks=true, linktoc=page, linkcolor=webgreen, citecolor=webgreen, urlcolor=webbrown]{hyperref}
\usepackage{datetime}
\usepackage[initials, alphabetic]{amsrefs}
\usepackage[dvipsnames]{xcolor}

\usepackage{geometry}
\usepackage{setspace}
\definecolor{webgreen}{rgb}{0,.4,0}
\definecolor{webbrown}{rgb}{.4,0,0}

\renewcommand{\dateseparator}{, }
\newcommand{\todaymy}{\shortmonthname ~ {\the\day}\dateseparator \the\year}

\newcommand{\Sn}[2]{
\begin{Bmatrix}
#1 \\ #2
\end{Bmatrix}
}

\makeatletter
\def\blfootnote{\xdef\@thefnmark{}\@footnotetext}
\makeatother

\newtheorem{Thm}{Theorem}[section]
\newtheorem{Lm}[Thm]{Lemma}
\newtheorem{Prop}[Thm]{Proposition}
\newtheorem{Cor}[Thm]{Corollary}
\newtheorem*{state}{Theorem}

\theoremstyle{definition}
\newtheorem*{ack}{Acknowledgements}
\theoremstyle{remark}
\newtheorem{Def}[Thm]{Definition}
\newtheorem{Not}[Thm]{Notation}

\newtheorem{Rem}[Thm]{Remark}
\newtheorem{convention}[Thm]{Convention}
\newtheorem{Exp}[Thm]{Example}
\numberwithin{equation}{section}

\newcommand{\<}{\langle}
\renewcommand{\>}{\rangle}

\DeclareMathOperator{\Li}{Li}
\DeclareMathOperator{\Arg}{Arg}

\DeclareMathOperator{\Real}{Re}
\DeclareMathOperator{\Imag}{Im}
\DeclareMathOperator{\Res}{Res}
\DeclareMathOperator{\Reg}{Reg}
\DeclareMathOperator{\Sing}{Sing}
\DeclareMathOperator{\id}{\mathbf{1}}
\DeclareMathOperator{\del}{\mathbf{d}}
\DeclareMathOperator{\E}{\mathbf{E}}
\DeclareMathOperator{\Sc}{\mathbf{S}}
\DeclareMathOperator{\T}{\mathbf{T}}

\def\H{\mathscr H}
\def\L{\mathscr L}

\def\D{\mathscr D}
\def\Tcal{\mathscr T}
\def\Del{\boldsymbol{\Delta}}
\def\Di{\mathbf{D}}
\def\Re{\mathbb R}
\def\Rat{\mathbb Q}

\def\Co{\mathbb C}

\def\Z{\mathbb Z}
\def\B{\mathbf{B}}
\def\A{\mathbf{A}}

\def\bb{\mathbf{b}}
\def\ba{\mathbf{a}}
\def\bc{\mathbf{c}}
\def\be{\mathbf{e}}
\def\dt{\mathbf{dt}}

\def\eps{\varepsilon}

\begin{document}
\title[]{Bernoulli Operators and Dirichlet Series}
\author[]{Bogdan Ion}
\address{Department of Mathematics, University of Pittsburgh, Pittsburgh, PA 15260}
\email{bion@pitt.edu}
\date{}
\blfootnote{Version: May 12, 2020
}

\begin{abstract}
We introduce and study some (infinite order) discrete derivative operators called  Bernoulli operators. They are associated to a class of power series (tame power series), which include power series that converge in the unit disk, have at most a pole singularity at $z=1$, and have analytic continuation to the unit disk centered at $z=1$ with possible isolated singularities of Mittag-Leffler type. We show that they all naturally act on, and take values into, the vector space of functions $f(s,t)$ in the image of the Laplace-Mellin transform that have (single valued) analytic continuation to the complex plane with possible isolated singularities. For $s$ in some right half-plane the action of the Bernoulli operator is given by a Dirichlet-type series and, as a consequence, such series acquire analytic continuation to the complex plane and allow a precise description of the singularities. For the particular case of $f(s,t)=t^s$, the action of the Bernoulli operators provide the analytic continuation of the Dirichlet series associated to tame power series. In this case, we record detailed information about the location of poles, their resides, and special values, as well as prove the uniqueness of tame Dirichlet series with specified poles, residues, and special values.
\end{abstract}

\maketitle

\thispagestyle{empty}

\section{Introduction}

One of the fundamental questions in the classical analytic theory of Dirichlet series deals with the identification of their maximal domain of convergence (or absolute convergence, or uniform convergence, or boundedness and regularity). A more difficult question is that of identifying the maximal domain of holomorphy and the nature of the singularities that might arise. In general, there is no clear phenomenology identified and there are indeed examples of Dirichlet series that exhibit a whole range of behavior with respect to analytic continuation and the associated singularities. 

On the other hand, the phenomena of holomorphic or meromorphic continuation is prominent for Dirichlet series of arithmetic origin. In the general theory of arithmetic Dirichlet series (see, e.g. \cites{Kac, Per, Sel} ) the constraints that are usually brought in as proxies for arithmetic are functional equations and, perhaps even more relevant in the light of the Davenport-Heilbronn example,  Euler  products. Both types of properties are, of course, present or expected for all L-functions that arise naturally and, as revealed through a lot of work over the past decades (particularly by Kaczorowski and Perelli) the set of series satisfying these properties seems to be rather small.

An immediate obstruction brought up by imposing early on requirements on the existence of functional equations and Euler products is that the series with these properties do not  form a vector space and, at least from an analytic point of view this might be desirable if one needs to embark on more systematic investigation into the phenomena of analytic continuation. In particular, the functional equations alone are typically enough to assure the holomorphic or meromorphic continuation of the Dirichlet series but aside from the examples of arithmetic origin the extent of this phenomenon seems rather limited. 

An important class of examples is provided by Hecke's  correspondence \cites{HecUbe, BK} between modular forms and Dirichlet series. In this situation the functional equation arises from modularity and the Euler product picks up the modular functions that are eigenvectors of the Hecke operators \cite{HecEP}. Even without imposing the modularity hypothesis, it suggests a relationship between the analytic properties (analytic continuation, singularities) of a generating function $\displaystyle \alpha(z)=\sum_{n\geq 0} a_{n+1}z^n$ (in Hecke's work $z=e^{2\pi i \tau}$ with $\tau$ in the upper half-plane) and those of the associated  Dirichlet series $\displaystyle \Di_\alpha(s)=\sum_{n\geq 1} \frac{a_n}{n^s}$. 

Another key observation is the fact that $\Di_\alpha(s)$ and $\displaystyle \Di_\alpha(s,t)=\sum_{n\geq 0} \frac{a_{n+1}}{(t+n)^s}$, $t>0$, have similar behavior with respect to analytic continuation, thus
studying $\Di_\alpha(s,t)$ would not lead to different phenomena. The following result \cite{MS}*{Theorem 3.1} is relevant in this respect. 
\begin{state}
Assume that $\Di_\alpha(s)$ is absolutely convergent for $\Real(s)>1$. Then
\begin{enumerate}[label={\roman*)}]
\item If $\Di_\alpha(s)$ extends to an entire function or a meromorphic function, so does $\Di_\alpha(s,t)$, respectively;
\item If $s=s_0$ is a pole for $\Di_\alpha(s,t)$ then $s=s_0+r$ is a pole for $\Di_\alpha(s)$ for some $0\leq \lceil -s_0 \rceil +1$.
\end{enumerate}
\end{state}

For us, the deformation with respect to the $t$ parameter will play a role; for example, it will enable us to differentiate with respect to $t$, and extend the focus from studying the analytic continuation of elements in the image of the Mellin transform to those in the image of the Laplace-Mellin transform.

We propose the study of the vector space  $\H$ consisting of functions $f(s,t)$, $(s,t)\in \Co\times \Re_+$, which arise as the (single valued) analytic continuation to $s\in \Co$, with possible isolated singularities, of functions of the type $$\frac{1}{\Gamma(s)}\int_{0}^\infty e^{-tu}\varphi(u)u^{s-1}du,\quad \Real(s)>1$$ 
with $\varphi(u)$ satisfying some mild technical condition. We refer to Definition \ref{def-H} for the details. The vector space $\H$ contains  trivial elements such as $t^s$ (corresponding to $\varphi(u)=1$) but it also contains elements corresponding to any  Dirichlet series that admits analytic continuation satisfying the conditions in the definition. It turns out that the elements of $\H$ necessarily have some appealing properties; for example  $\H$ is closed under derivation with respect to $t$ and the set of singularities of an element of $\H$ must be stable under translation with respect to $\Z_{\geq 0}$. Also, if $n\in \Z_{\geq 0}$ is not a singular point then $f(n,t)$ is a polynomial in $t$ of degree at most $n$. 

On the other hand, we consider $\Tcal$ the vector space of tame power series (see  Definition \ref{def-ts}): power series $\alpha(z)$ convergent in the unit disk, with $z=1$ at most a pole (of order $\nu\geq 0$), and such that $(z-1)^\nu \alpha(z)$ has a muti-power series expansion (see Definition \ref{def-mps}) around $z=1$  that converges absolutely and uniformly in a neighborhood of the interval $(0,1]$. A large class of examples of such series consists of analytic functions which, aside from being holomorphic in the unit disk with at most a pole singularity at $z=1$, have at most isolated singularities of Mittag-Leffler type in the disk of radius $1$ centered at $z=1$. It appears possible that $\Tcal$ can be extended (through the use of Puisseux series) to include series that have a possible algebraic branch point at $z=1$.

To any tame $\alpha(z)$ we associate a discrete difference operator $\B_\alpha$ (of infinite order) that we call  Bernoulli operator. Our motivation for considering such operators stems from considerations related to the (group theory) renormalization of divergent series \cite{CI}. The first occurrence of such an operator and its use to describe the analytic continuation of the Riemann zeta function goes back to Hasse \cite{HasEin} whose inspiration seems to have been Worpitzky's formula \cite{WorStu} for Bernoulli numbers. A similar relationship with the zeta function was established earlier by Ser \cite{SerSur}; the equivalence of the formulas of Hasse and Ser is presented in \cite{BlaThr}*{Theorem 1}. 

Our main result, Theorem \ref{t2m}, shows that $\H$ lies inside the domain of any Bernoulli operator and, in fact,  we have $\B_\alpha: \H\to \H$. This endows $\H$ with a $\Tcal$-module structure. Furthermore, $f(s,t)$ and $\B_\alpha f(s,t)$ have the same singular locus. The connection between the action of the operator $\B_\alpha$ and the Dirichlet-type series 
$$
\Di_\alpha^f(s,t)=\sum_{n=0}^\infty {a_{n+1}}{f(-s,t+n)}
$$
is the following. For $f(s,t)\in \H$ for which $\varphi(u)=o(u^\gamma)$, $u\to 0^+$, for some $\gamma>-1$ we have that  $\B_\alpha f(-s,t)$ is the analytic continuation of 
$$
s\cdot(s+1)\cdots(s+\nu-1)\Di_\alpha^f(s+\nu,t).
$$
This shows, in particular, that the analytic continuation for $\Di_\alpha^f(s,t)$ stems directly from the same property for $f(s,t)$ and in fact $\Di_\alpha^f(s,t)$ can only acquire at most $\nu$ simple poles (at fixed locations) on top of the singular locus of $f(s,t)$. The values of $\Di_\alpha^f(s,t)$ at negative integers (which are not singular points) can be computed rather directly because $f(n,t)$ are polynomials and on polynomial functions of bounded degree $\B_\alpha$ acts as a finite order difference operator. In fact, on polynomial functions, the action of $\B_\alpha$ matches the action of another operator $\T_\alpha(-\del)$ associated to $\alpha(z)$ which we call Todd operator. The Todd operator is an infinite order differential operator obtained by substituting the derivative $\del$ with respect to $t$ for $u$ in the Taylor expansion around $u=0$ for $(-u)^\nu\alpha(e^u)$. On polynomial functions of bounded degree $\T_\alpha(-\del)$ acts as a differential operator of finite order and its direct relationship with $\alpha(z)$ makes the computation of its action straightforward. The terminology is justified by fact that for $\alpha(z)=1/(1-z)$ the operator $\T_\alpha(\del)$ is precisely the Todd operator in \cites{KhPu, BV} that appears in the context of the combinatorial Riemann-Roch theorem and lattice point counting in simple polytopes.

In Section \ref{dirichlet} we provide more comprehensive details of the consequences of these results for $t^s\in \H$. In this case, $\Tcal\cdot t^s$ consists of Dirichlet series associated to tame power series. Their analytic continuation is a direct consequence of the fact that $t^s$ is entire. In this case, $\Di_\alpha(s,t)$ is meromorphic with possible simple poles at $s=1,2,\dots,\nu$. Theorem \ref{t3} and Corollary \ref{cor-t3} provide full details on which poles arise (depending on the value of $t$), their residues, and the special values $\Di_\alpha(-n,t)$, $n\geq 0$. All information can be read from the Bernoulli polynomials defined as $B_\alpha[n,t]=\B_\alpha t^n=\T_\alpha(-\del) t^n$. In particular,  the exponential generating function for the Bernoulli polynomials is shown to be
$$
\sum_{n=0}^{\infty} B_\alpha[n;t]\frac{u^n}{n!}=(-u)^\nu \alpha(e^u)e^{tu}.
$$
We show in Theorem \ref{t6} that all these results hold without the hypothesis that $\alpha(z)$ has a multi-power series expansion around $z=1$. For the computation of the special values we rely on \cite{EveMer} where in fact  more detailed results on the analytic continuation of $\Di_\alpha(s)$ have been obtained under some technical hypothesis, in a fashion that also allows the computation of poles, residues, and special values. We refer to \S\ref{ow} for some of the details.

One interesting outcome of this analysis is a uniqueness theorem (within $\Tcal\cdot t^s$) for Dirichlet series. As it turns out,  $\Di_\alpha(s,t_0)$ is uniquely determined by a finite subset of $\Z_{>0}$ (the location of the simple poles), a corresponding set of non-zero complex numbers (the residues at those poles), and a sequence of complex numbers (the special values $\Di_\alpha(-n,t_0)$). We refer to Theorem \ref{t5} for the details. The question of existence of a (unique) Dirichlet series determined by this data, points to a possible converse theorem that is quite different from the classical converse theorems for L-functions \cites{Ham1, Ham2, Ham3, HecUbe, Wei, KMP} but closer to the Beurling converse theorem \cite{Beu} for the zeta function (see also \cite{Dix}). It also suggests that the information that is typically sought for from an L-function (poles, residues, special values) is sufficient to identify it  uniquely. We conclude the article with a set of old and new examples illustrating our results (Section \ref{sec-exp}) and a discussion of what can be expected from power series $\alpha(z)$ with different types of singularities at $z=1$.

The $\Tcal$-module structure on $\H$ (or possible variations) seems to allow for some perspective into the phenomena of formation of singularities in analytic continuation. It is also amenable to the multi-variable situation and, in particular, to the corresponding phenomena in the context of analysis on symmetric cones. We hope to pursue some of these directions in the future.

 \begin{ack}
 I thank Gunduz Caginalp for his questions and discussions on the regularization of divergent series and for his general interest in this work. I thank Jeffrey Lagarias for bringing to my attention and making available the 2018 University of Michigan dissertation of Corey Everlove \cite{EveMer} which, among other things, contains some general results about analytic continuation of Dirichlet series. The relevant facts are briefly discussed in \S\ref{ow}. This work was partially supported by the Simons Foundation grant 420882. 
 \end{ack}
 
\section{Notation}

\subsection{} Throughout, we reserve $u$ and $t$ to denote real variables with domain $\Re_{+}=(0,\infty)$; accordingly, $du$ and $dt$ refer to the Lebesgue measure on $\Re_{+}$. All integrals with respect to these measures are Lebesgue integrals of real or complex-valued measurable functions. All spaces of functions that will be considered, in particular the domains of all operators, are based on functions on $\Re_+$. Similarly, we reserve $z$ and $s$ to denote complex variables with domain $\Co$, unless otherwise specified. We will make use of the Gamma function $\Gamma(s)$ and the falling and raising factorials $s^{\underline{n}}=\Gamma(s+1)/\Gamma(s-n+1)$, $s^{\overline n}=\Gamma(s+n)/\Gamma(s)$, $n\in \Z$.

\subsection{}  For integers $0\leq k\leq n$ let $\Sn{n}{k}$ denote the corresponding Stirling number of the second kind. They count the number of set partitions of the set $\{1,2,\dots,n\}$ into $k$ (non-empty) parts. For us, their relevance rests on their occurrence in the Fa\` a di Bruno formula for the composition of the formal series 
$$
\varphi(u)=\sum_{n\geq 0} \frac{\varphi_n}{n!}u^n \quad \text{and} \quad e^u-1=\sum_{n\geq 1} \frac{1}{n!}u^n.
$$
More precisely, 
\begin{equation}\label{eq14}
\varphi(e^u-1)=\sum_{n\geq 0} \frac{\psi_n}{n!}u^n, \quad \text{where}\quad \psi_0=\varphi_0 ~\text{and}~ \psi_n=\sum_{k=1}^n \Sn{n}{k}\varphi_k, ~n\geq 1.
\end{equation}
We refer to \cite{StaEnu}*{Ch. 5, Thm 5.1.4} for the proof of the composition formula in full generality.

\subsection{} 
For $s_0\in \Co$ we denote left open half space, the right open half-space, and the right closed half-space determined by the line $\Real(s)=\Real(s_0)$ by
 $$
 \Co_{s_0^-}=\{s\in \Co~|~\Real(s)<\Real(s_0)\}, \quad 
 \Co_{s_0^+}=\{s\in \Co~|~\Real(s)>\Real(s_0)\}, \quad \text{and}\quad 
 \Co_{s_0}=\overline{\Co_{s_0^+}}.
 $$

\subsection{} Fix $N\in \Z_{>0}$. We will use bold letters to denote vectors in $\Co^N$. The components of $\mathbf{z}\in \Co^N$ are denoted by $z_i$, $1\leq i\leq N$ and $|\mathbf{z}|_p$ refers to the usual $\ell^p$-norm, $1\leq p\leq \infty$. We will abreviate $|\mathbf{z}|_1$ to $|\mathbf{z}|$.
If $n\in \Z$, we denote by ${\mathbf{n}}$ the vector in $\Co^n$ with all components equal to $n$. We will use $\<\cdot,\cdot\>$ to denote the usual bilinear scalar product on $\Co^N$. 

\subsection{} \label{notation-vectors}
For $\mathbf{z}\in \Co^N$ and 
$\ba \in \Z_{\geq 0}^N$ we set
$$
\mathbf{z}^\ba=z_1^{a_1}\cdots z_N^{a_N}.
$$
The factors for which $a_i=0$ are not included in the product.

If $\ba$ and $\bb\in \Re^N$ and $a_i\leq b_i$ for all $1\leq i\leq N$, we denote 
$$
[\ba,\bb]=[a_1,b_1]\times\cdots\times[a_N,b_N].
$$
For $\bc \in \Z_{\geq 0}^N$, we use $[\ba,\bb]^\bc$ to refer to 
$$
[a_1,b_1]\times \cdots \times[a_1,b_1]\times\cdots \times [a_N,b_N]\times \cdots\times[a_N,b_N]
$$
where the interval $[a_i, b_i]$ appears $c_i$ times. The factors for which $c_i=0$ are not included in the direct product. 
\subsection{} For $\ba=(a_1,\dots,a_N)\in \Co^N$ and $ \mathbf{r}=(r_1,\dots,r_N)\in \Re^N$, denote by 
$D(\ba,\mathbf{r})$ the open polydisk 
$$
\{
(z_1,\dots,z_N)~|~|z_i-a_i|<r_i, ~1\leq i\leq N 
\}.
$$

A function $F(\mathbf{z})$ which is holomorphic around a point $\mathbf{z_0}$ of its domain has a power series expansion 
$$
F(\mathbf{z})=\sum_{\mathbf{i}\in \Z^{N}_{\geq 0}} c_{\mathbf{i}} (\mathbf{z}-\mathbf{z_0})^\mathbf{i}
$$
which converges absolutely and uniformly in a polydisk $D(\mathbf{z_0},\mathbf{r})$. For $N\geq 2$, the largest domain of convergence of a power series is typically larger than a polydisk (a logarithmically convex Reinhardt domain). We denote by $|F|(\mathbf{z})$ the series given by the absolute value of the terms in $F(\mathbf{z})$.

\subsection{} To a sequence $(a_n)_{n\geq 1}$ of complex numbers, we associate the generating series $$\alpha(z)=\sum_{n=0}^\infty a_{n+1}z^n$$ and the Dirichlet series
$$
\Di_\alpha(s,t)=\sum_{n=0}^\infty \frac{a_{n+1}}{(t+n)^s}.
$$
The corresponding classical Dirichlet series 
$$
\Di_\alpha(s)=\sum_{n=1}^\infty \frac{a_{n}}{n^s}
$$
is obtained as  $\Di_\alpha(s,1)$.

From the general theory of Dirichlet series, the domain of absolute convergence of $\Di_\alpha(s)$ with respect to $s$ is not empty if and only if  $a_n=O(n^c)$ for some $c>0$. In such a case, the radius of convergence of $\alpha(z)$ is at least $1$. Therefore, for the purpose of studying the Dirichlet series associated to $(a_n)_{n\geq 1}$ we should restrict to sequences for which the series $\alpha(z)$ converges in $D(0,1)$.

\subsection{} More generally, we can consider the series 
$$
\Di^f_\alpha(s,t)=\sum_{n=0}^\infty {a_{n+1}}{f(-s,t+n)},
$$
associated to the sequence $(a_n)_{n\geq 1}$  and a function $f(s,t)$ that is holomorphic in $s\in \Co$. We postpone  (until \S\ref{dio} and \S\ref{AC}) describing sufficient conditions on $f(s,t)$ that assure the  convergence of such series, at least for $s$ with sufficiently large real part.

\subsection{} We use $\id$ to denote the identity operator, 
$\del$ to denote the derivative operator $\displaystyle \frac{d}{dt}$, $\E$ to denote the (forward) shift operator 
\[
\E f(t)=f(t+1)
\]
and $\Del$ to denote the discrete (forward) derivative operator $\Del=\E-\id$,
\[
\Del f(t)=f(t+1)-f(t).
\]

\subsection{}
In fact, for any fixed  $h>0$, we consider the corresponding operators:  $\Sc_h$ the scaling operator defined as
\[
\Sc_h f(t)=f(ht),
\]
$\E_h$ the (forward) shift operator defined by 
\[
\E_h f(t)=f(t+h),
\] and the difference  operator $\Del_h=\E_h-\id$. 
Note that, for any $h>0$, we have 
\[\
\E_h=\Sc_{h^{-1}}\E \Sc_{h},\quad \text{and} \quad \Del_h=\Sc_{h^{-1}}\Del \Sc_{h}.\]
In particular, all analytic properties of $\Del_h$ are inherited from those of $\Del$.

\subsection{} Let  $\mathbf{h}\in \Re^N_{>0}$ and $\mathbf{i}\in \Z^N_{\geq 0}$. By analogy with the notation set in \S\ref{notation-vectors} let $$\E_{\mathbf{h}}=\E_{h_1}\cdots \E_{h_N}=\E_{h_1+\cdots+h_N}\quad \text{ and}\quad  \E_{\mathbf{h}}-\id=(\E_{h_1}-\id)\cdots (\E_{h_N}-\id)=\Del_{h_1}\cdots\Del_{h_N}.
$$
In particular,
$$
\Del_{\mathbf{h}}^{\mathbf{i}}=(\E_{\mathbf{h}}-\id)^{\mathbf{i}}=\Del_{h_1}^{i_1}\cdots \Del_{h_N}^{i_N}.
$$
For $\mathbf{u}\in \Re^{N}_{>0}$, $\mathbf{du}$ will refer to the Lebesgue measure $du_1\cdots du_{N}$.


\section{The Laplace transform}  \label{LT}

\subsection{}  Let $\mu$ denote a $\Co$-valued function of bounded variation on all intervals of the form $(0,R),~R>0$. We denote by  $\L(d\mu)$  the Laplace-Stieltjes transform of $\mu$, whose values are given by the following Riemann-Stieltjes integral 
\begin{equation}\label{eq2}
\L(d\mu)(s)=\int_0^\infty e^{-su}d\mu(u),
\end{equation}
for $s\in \Co$ for which the integral converges. We refer to \cite{WidLap} for a thorough treatment of the theory of the Laplace transform in this generality. For certain domains for $\L$, descriptions of the image are classically known \cite{WidLap}*{VII, \S12-17}.

\subsection{}
 If the integral \eqref{eq2} converges for $s_0$ then it converges locally uniformly for  $s$ in the open half-plane 
 $$
 \Co_{s_0^+}=\{s\in \Co~|~\Real(s)>\Real(s_0)\}.
 $$
Furthermore, $\L(d\mu)(s)$ is holomorphic in $\Co_{s_0^+}$ and  for any $k\geq 0$ its $k$-th derivative is given by 
\begin{equation}\label{eq2b}
\L(d\mu)^{(k)}(s)=\int_0^\infty e^{-su}(-u)^kd\mu(u).
\end{equation}
If the integral \eqref{eq2} converges absolutely for $s_0$ then it converges uniformly and absolutely in the closed half-plane 
 $ \Co_{s_0}$.

\subsection{}

If $\L({d\mu})(s)$ is defined and $s\in \Co_{0^+}$ then \cite{WidLap}*{II, Thm2.3a}
$$
\L(d\mu)(s)=s \int_0^\infty e^{-su}\mu(u)du-\mu(0).
$$
We will be interested in the situation for which $\L(d\mu)(s)$ is defined (at least) in the open half-plane $ \Co_{0^+}$ and without loss of generality we can consider only the Lagrange-Lebesgue transform. 

\subsection{}

The Laplace-Lebesgue transform (or simply the Laplace transform) $\L(\varphi)=\L(\varphi(u)du)$  of $\varphi$ is defined by 
\begin{equation}\label{eq2c}
\L(\varphi)(s)=\int_0^\infty e^{-su}\varphi(u)du
\end{equation}
for $s\in \Co$ for which the integral converges. As its domain we will consider $\D(\L)$, the $\Co$-vector space of functions $\varphi$ which are integrable on intervals $(0,R)$ for every $R>0$, and for which the integral \eqref{eq2c} is  \emph{absolutely} convergent for all $s\in \Co_{0^+}$. We denote by $\Imag(\L)$ the image of $\L$ on this domain. Note that $\varphi\in \D(\L)$ implies that $\varphi$ is locally integrable on $(0,\infty)$. Although holomorphic in $\Co_{0^+}$, we will mostly consider $\L(\varphi)$ as a function $\L(\varphi)(t)$ with $t\in\Re_+$.

\subsection{} We define the Laplace-Mellin transform of $\varphi(u)\in \D(\L)$ as the function
$$
\L(\varphi(u)u^{s-1}/\Gamma(s))(t)
$$
as a function of two arguments $(s,t)\in \Co_{1^+}\times \Re_{>0}$.  In other words, the Laplace-Mellin transform of $\varphi(u)$ for fixed $t$ is the Mellin transform of $e^{-tu}\varphi(u)/\Gamma(s)$. The Mellin transform of $\varphi(u)/\Gamma(s)$ would correspond to evaluation of the Laplace-Mellin transform at $t=0$.

\subsection{} 

For $\varphi, \psi$ defined on $\Re_{+}$ their convolution is defined by
$$
(\varphi*\psi)(u)=\int_0^u\varphi(u-x)\psi(x)dx=\int_0^u\psi(u-x)\varphi(x)dx.
$$
We will use the following classical result \cite{WidLap}*{II, Thm. 12.1a}.
\begin{Thm}\label{convolution}
Let $\varphi, \psi\in \D(\L)$. Then $\varphi*\psi\in\D(\L)$  and 
\begin{equation}
\L(\varphi*\psi)=\L(\varphi)\L(\psi). 
\end{equation}
In particular, $\Imag(\L)$ is closed under product.
\end{Thm}
For $s\in \Co_{0+}$ we have $u^{s-1}\in \D(\L)$ and
$$
\L(u^{s-1})=t^{-s}\Gamma(s).
$$
For $s_1,s_2\in \Co_{0^+}$, the convolution of $u^{s_1-1}/\Gamma(s_1)$ and $u^{s_2-1}/\Gamma(s_2)$ is $u^{s_1+s_2-1}/\Gamma(s_1+s_2)$.

\subsection{} For $s_0>0$, $s\in \Co_{s_0^{-}}\cap\Co_{0^+}$ denote
\begin{equation}\label{eq-C}
C_{s,s_0} = \frac{\sinh(\pi \Imag(s))}{\pi \Imag(s)}\sqrt{\left(1+\frac{\Imag(s)^2}{\Real(s)^2}\right)\left(1+\frac{\Imag(s)^2}{\Real(s_0-s)^2}\right)}.
\end{equation}
Note that $C_{s,s_0}$ is a positive constant; its dependence on $s,s_0$ is locally uniformly continuous.
For later use we record the following

\begin{Lm}\label{lm-bd1}
Let $\psi\in \D(\L)$ such that $\psi$ is real-valued, non-negative, and increasing. Then, for $s_0>0$, $s\in \Co_{s_0^{-}}\cap\Co_{0^+}$ we have
$$
\left|\L(\psi(u)u^{s-1}/\Gamma(s))(t)\right|\leq C_{s,s_0}t^{\Real(s_0-s)}\L(\psi(u)u^{s_0-1}/\Gamma(s_0))(t).
$$
\end{Lm}
\begin{proof} We have
$$
t^{s-s_0}\L\left(\psi(u)\frac{u^{s-1}}{\Gamma(s)}\right)(t)=\L\left(\frac{\psi(u)u^{s-1}}{\Gamma(s)}*\frac{u^{s_0-s-1}}{\Gamma(s_0-s)}\right)(t).
$$
Therefore, it is enough to argue that 
$$
\left|\frac{\psi(u)u^{s-1}}{\Gamma(s)}*\frac{u^{s_0-s-1}}{\Gamma(s_0-s)}\right|\leq   C_{s,s_0}\psi(u)\frac{u^{s_0-1}}{\Gamma(s_0)},
$$
for some $C_{s,s_0}$ as specified by the statement.

From the definition of convolution and the fact that $\psi$ is non-decreasing we obtain that
$$
\left|\frac{\psi(u)u^{s-1}}{\Gamma(s)}*\frac{u^{s_0-s-1}}{\Gamma(s_0-s)}\right|\leq \psi(u)\frac{u^{s_0-1}}{\Gamma(s_0)} \cdot \frac{\Gamma(\Real(s))}{|\Gamma(s)|} \cdot \frac{\Gamma(s_0-\Real(s))}{|\Gamma(s_0-s)|}.
$$
We use the following well-know expressions \cite{AR}*{(5.8.3)}, \cite{RO}*{(4.36.1)}
$$\left(\frac{\Gamma(x)}{|\Gamma(x+iy)|}\right)^2=\prod_{k=0}^\infty\left(1+\frac{y^2}{(x+k)^2}\right),\quad x,y\in \Re, x\not\in \Z_{\leq 0}$$
and
$$
\sinh(\pi z)=\pi z\prod_{n=1}^\infty(1+z^2/n^2),\quad z\in \Co,
$$
to deduce that
$$
\left(\frac{\Gamma(x)}{|\Gamma(x+iy)|}\right)^2=\left(1+\frac{y^2}{x^2}\right)\prod_{k=1}^\infty\left(1+\frac{y^2}{(x+k)^2}\right)\leq \left(1+\frac{y^2}{x^2}\right) \frac{\sinh(\pi y)}{\pi y}, \quad x,y\in \Re, x\not\in \Z_{\leq 0}.
$$
Applying this for $x+iy=s$ and $x+iy=s_0-s$ we obtain 
$$
 \frac{\Gamma(\Real(s))}{|\Gamma(s)|} \cdot \frac{\Gamma(s_0-\Real(s))}{|\Gamma(s_0-s)|}\leq C_{s,s_0}.
$$
\end{proof}
\begin{Rem}\label{rem-bd1}
It is important to note that, for any $n\geq 0$, we have $$C_{s,s_0}\geq C_{s+n,s_0+n}.$$
\end{Rem}



\section{Discrete integral operators}\label{dio}

\subsection{}\label{series-ad}

Let $(a_n)_{n\geq 1}$ be a sequence of complex numbers. We will denote by $\alpha(z)=\sum_{n=0}^\infty a_{n+1}z^{n}$ the associated generating series (as a formal power series) and by $\A_\alpha$ the operator
$$
\A_\alpha=\sum_{n=0}^\infty a_{n+1}\E^{n}.
$$
This operator can be regarded as the discrete integral (i.e. summation) operator  associated to the measure on $\Z_{\geq 0}$ associated to the sequence $(a_{n+1})_{n\geq 0}$. The natural context for such considerations is that of a functional calculus for the operator $\E$. If we consider $\E$ as bounded operator acting on square-integrable functions, its operator norm is $1$, and its spectrum is the unit circle centered at zero. (This can be seen from the fact that  $\E_h$ and the multiplication operator by the function $e^{it}$ are conjugated by a unitary operator -- the Fourier transform.) From this point of view, $\A_\alpha$ is the operator associated through holomorphic functional calculus to $\alpha(z)$, and whether such a construction is justified depends on $\alpha(z)$ being holomorphic in a neighborhood of the unit circle (the spectrum of $\E$). However, we would like to consider $\A_\alpha$  for which the unit circle is the boundary of the disk of convergence of $\alpha(z)$. For this reason we consider $\A_\alpha$ as an operator with domain the vector space $\D(\A_\alpha)$  of  $\Co$-valued functions $f(t)$ on $\Re_+$  for which the series
\begin{equation}\label{Aseries-def}
\sum_{n=0}^\infty a_{n+1}\E^{n}f(t)=\sum_{n=0}^\infty a_{n+1}f(t+n)
\end{equation}
converges absolutely and locally uniformly in $t$.  In general, we do not expect that $\A_\alpha$ is continuous with respect to any natural topology on its domain, but $\A_\alpha$ will preserve local integrability and continuity of the argument.

\subsection{}  For $\mu$ a regular Borel measure on $\Re_+$,  we denote by $L^1_\infty(\Re_+, d\mu)$ the space of functions that are absolutely integrable with respect to $\mu$ in a neighborhood of $+\infty$.  Let $A$ be the function defined by $A(t)=a_{n+1}, \quad \text{if} \quad t\in [n,n+1)$, $n\geq 0$. Denote by $\mu_\alpha$ the measure $A(t)\dt$ and let  $f\in L^1(\Re_+,d\mu_\alpha)$. The function $\A_\alpha f(t)$ 
is a.e. finite and  in $L^1_{\rm loc}(\Re_+)$. Indeed,
\[
\sum_{n=0}^\infty |a_{n+1}|\int_0^1 |f(t+n)|\dt=\int_0^\infty |A(t)f(t)|dt<\infty,
\]
and, by the Fubini-Tonelli theorem, 
\[
\int_0^1 |\A_\alpha f(t)|dt=\int_0^1 \sum_{n=0}^\infty  |a_{n+1} f(t+n)|\dt<\infty.
\]
The convergence of the series \eqref{Aseries-def} depends only on the  behavior of the function in a neighborhood of $+\infty$, so the hypothesis that $f\in L^1(\Re_+, d\mu_\alpha)$ can be replaced by  $f\in L^1_\infty(\Re_+,d\mu_\alpha)$.


\subsection{} In general, aside from the above remarks, not much can be inferred about the domain on $\A_\alpha$. We can say slightly more in the case when $\alpha(z)$ has some convergence properties, specifically, when  $\alpha(z)$  converges in $D(0,1)$ and has (at most) a pole at $z=1$.

\begin{Rem}\label{rem2}
Assume that $\alpha(z)$  converges in $D(0,1)$. Let $f(t)$ such that for some continuous $\Z$-periodic function $c(t)$ taking values in $D(0,1)$ we have
\[
\E f(t)=c(t) f(t).
\]
Then, the series \eqref{series-def} is absolutely convergent and $f(t)\in \D(\A_\alpha)$.
\end{Rem}

\begin{Exp} An example of this type is the exponential function $$f(t)=a^t.$$ If $a\in (0,1)$ then the convergence of $\A_\alpha(a^t)$ follows from the convergence of the series $\alpha(z)$ and
$$
\A_\alpha(a^t)= a^t\alpha(a).
$$
\end{Exp}

\subsection{} Based on this example we can show that depending on the growth of $\alpha(z)$ at $z\to 1^-$ some part of the $\Imag{\L}$ is included in $\D(\A_\alpha)$.

\begin{Prop}\label{p2}
Assume that $\alpha(z)$  converges in $D(0,1)$ and has at most a pole at $z=1$ (of order $\nu\geq 0$).
Let $f=\L(\varphi)\in \Imag(\L)$ and $\varphi(u)=o(u^{\gamma+\nu})$, $u\to 0^+$, for some $\gamma>-1$. Then $f\in \D(\A_\alpha)$ and 
$$
\A_\alpha f(t)=\L(\alpha(e^{-u})\varphi(u)).
$$
\end{Prop}

\begin{proof} We have $\E^{n}f(t)=\L(e^{-un}\varphi(u))$. For $u\in \Re_+$ the series 
$$
\sum_{n=0}^\infty a_{n+1}e^{-un}\varphi(u)
$$
converges to $\alpha(e^{-u})\varphi(u)$ pointwise. Since $\alpha(e^{-u})=O(u^{-\nu})$, $u\to 0+$, and $\alpha(e^{-u})=O(1)$, $u\to +\infty$,  we obtain that $\alpha(e^{-u})\varphi(u)\in \D(\L)$. By the Dominated Convergence Theorem we obtain that the series $\A_\alpha f$ converges locally uniformly to $\L(\alpha(e^{-u})\varphi(u))$. 
\end{proof}
We will examine an infinite order difference operator with larger domain that extends the action of $\A_\alpha$. This operator is related to the Laurent expansion of $\alpha(z)$ around $z=1$. Some sufficient conditions on (the analytic continuation of) $\alpha(z)$ that make the definition possible are discussed in the next section.


\section{Multi-power series}

\subsection{}
For $\be\in \Z^N_{>0}$ let $\lambda_\be: \Co\to \Co^N$ defined by $\lambda_\be(z)=(z^{e_1},\dots,z^{e_N})$.  Remark that $\lambda_\be(0)={\mathbf{0}}$ and $\lambda_\be(1)={\mathbf{1}}$.
\begin{Def}\label{def-mps}
A multi-power series around $z_0\in \Co$ is a series of the form
$$
F(\lambda_\be(z))=\sum_{\mathbf{i}\in \Z^N_{\geq 0}} c_{\mathbf{i}} (\lambda_\be(z)-\lambda_\be(z_0))^\mathbf{i}
$$
for some $N\in \Z_{>0}$, $\be\in \Z^N_{>0}$, and power-series expansion $
F(\mathbf{z})=\sum_{\mathbf{i}\in \Z^{N}_{\geq 0}} c_{\mathbf{i}} (\mathbf{z}-\mathbf{z_0})^\mathbf{i}
$. We use the usual notion of  (absolute, uniform) convergence for such series. 
\end{Def}

A multi-power series around ${\mathbf{0}}$ takes the form 
$
\displaystyle \sum_{\mathbf{i}\in \Z^N_{\geq 0}} c_{\mathbf{i}} z^{\<\be, \mathbf{i}\>}
$
and in this case the notion of multi-power series coincides with notion of  power series around $0$. If $F(\mathbf{z})$ represents a holomorphic function around the point $\lambda_\be(z_0)$ the convergence of its power series expansion implies the convergence of the multi-power series $F\circ \lambda_\be(z)$  around $z_0$. For fixed $N$, $\be$ and $\alpha(z)$ defined in a neighborhood of $0$ there is at most one $F(\mathbf{z})$ such that $F(\mathbf{z})$ is holomorphic in a neighborhood of $\mathbf{0}$ and  $F(\lambda_\be(z))=\alpha(z)$ but, two multi-power series can, in principle, represent the same function.

\begin{Def} \label{def-ts}
A power series around $0$ is said to be {\it tame} if it represents a function $\alpha(z)$ of one complex variable with the following properties
\begin{itemize}
\item $\alpha(z)$ is holomorphic in $D(0,1)$;
\item $\alpha(z)$ has a pole at $z=1$ (of order $\nu\geq 0$); 
\item $(z-1)^\nu \alpha(z)$ has a multi-power series expansion $C_\alpha(z)$ around $1$ that converges absolutely and uniformly in a neighborhood of the interval $(0,1]$.
\end{itemize}
We use $\Tcal$ to denote the space of tame power series around $0$.
\end{Def}
Note that  the multi-power series $C_\alpha(z)$ from the definition is not necessarily unique.
\begin{Rem}
It is clear from the definition that $\Tcal$ is in fact a $\Co$-algebra, the multiplication being the usual multiplication of power series.
\end{Rem}

As we point out in what follows,  a class of examples of tame power series arise from functions that are holomorphic in $D(0,1)$ and have certain isolated singularities in $D(1,1)$.

\subsection{} Let us recall the following theorem of Mittag-Leffler \cite{M-LSur} on the existence of functions with prescribed singularities. We emphasize that the singularities of the function $f$ in the statement can be poles and certain essential singularities but not branch points.

\begin{Thm} Let $U\subset \Co$ be an open set and $S \subset U$ a discrete subset. For each point $q\in S$, we are given a holomorphic function $\alpha_q$ on $U\setminus\{q\}$. Then there exists a holomorphic function $\alpha$ on $U\setminus S$ such that for each $q\in S$, the function $\alpha-\alpha_q$ has a removable singularity at $q$.
\end{Thm}

 This motivates the following definition.
\begin{Def}
Let Let $U\subset \Co$ be an open set, $S \subset U$ a discrete subset and $\alpha(z)$ a holomorphic function on $U\setminus S$. We say that $\alpha(z)$ is of Mittag-Leffler  type if for any $q\in S$ there exist \emph{entire} functions $\alpha_q(z)$ such that  the function
$$
\alpha(z)-\alpha_q\left(\frac{1}{z-q}\right)
$$
has a removable singularity at $q$. We say that $\alpha(z)$ is of finite Mittag-Leffler type if, in addition, $\alpha(z)$ has finitely many singularities in $U$.
\end{Def}

\subsection{} While a holomorphic function $\alpha(z)$ of one variable has a power series expansion around any point in its domain, we would be interested in the situation when we represent is by a multi-power series expansion with a larger domain of convergence than typically guaranteed by a power series. For the purpose of this note, we are specifically interested in functions the are holomorphic in $\Omega=D(0,1)\cup D(1,1)$ except for possibly finitely many singularities.

\begin{Prop}\label{prop-ML}
 Let $\alpha(z)$ be a function of one complex variable with the following properties
\begin{itemize}
\item $\alpha(z)$ is of finite Mittag-Leffler type in $\Omega$;
\item $\alpha(z)$ is holomorphic in $D(0,1)$;
\item $\alpha(z)$ has a pole at $z=1$ (of order $\nu \geq 0$).
\end{itemize}
Then, there exist $m\geq 1$ such that $(z-1)^\nu \alpha(z)$ has a  multi-power series expansion $C_\alpha(z)$ around  $1$ that converges  absolutely and uniformly on 
$$
\Omega_m=\{z ~|~ |z|<\sqrt[m]{2}, ~|\Arg(z)|<\frac{\pi}{2m}\}.
$$
In particular, the Taylor expansion of $\alpha(z)$ around $0$ is tame.
\end{Prop}

\begin{proof} We can assume that $z=1$ is a regular point. Let $S$ be the set of singularities of $\alpha(z)$ in $\Omega$. Let $N=|S|+1$ and label the elements of $S=\{q_2,\dots,q_N\}$. For each $q_i$ we have 
$$
|q_i|\geq 1\quad \text{and} \quad |\Arg(q_i)|<\pi/3.
$$
Since $q_i\neq 1$, we can find $e_i\in \Z_{>1}$ such that 
$$
|q_i|\geq 2\quad\quad \text{or} \quad \quad |q_i|=1~\text{and}~e_i|\Arg(q_i)|\geq \pi/3.
$$
Set $e_1=1$ and let $\be=(e_1,e_2,\dots,e_N)\in \Z^N_{>0}$. Let $\displaystyle m=\max_{1\leq i\leq N}{e_i}$.

The function $f(z)$ can be written as 
$$
g(z)+\sum_{i=2}^N g_i\left(\frac{1}{z-q_i}\right)
$$
with $g(z)$ holomorphic in $\Omega$, and $g_i(z)$ entire. We will  write  it as
$$
\alpha(z)=g(z)+\sum_{i=2}^N g_i\left(\frac{k_i(z)}{z^{e_i}-q_i^{e_i}}\right),
$$
where $\displaystyle k_i(z)=\frac{z^{e_i}-q_i^{e_i}}{z-q_i}$ are polynomials.

The function 
$$
F(\mathbf{z})=g(z_1)+\sum_{i=2}^N g_i\left(\frac{k_i(z_1)}{z_i-q_i^{e_i}}\right),
$$
is holomorphic in $D({\mathbf{0}},{\mathbf{1}})\cup D({\mathbf{1}},{\mathbf{1}})\subset \Co^N$.

Furthermore,
$
\alpha(z)=F\circ \lambda_\be(\mathbf{z})
$
on  $$U=\{z\in D(0,1)\cup D(1,1) ~|~ \lambda_\be(z)\in D({\mathbf{0}},{\mathbf{1}})\cup D({\mathbf{1}},{\mathbf{1}})\}.$$ The holomorphic function $F(\mathbf{z})$ has a power series expansion around ${\mathbf{1}}$ that is convergent in 
$D({\mathbf{1}},{\mathbf{1}})$. This power series expansion induces a multi-power series expansion $C_\alpha(z)=F(\lambda_\be(\mathbf{z}))$ for $\alpha(z)$. The domain of convergence of this multi-power series expansion is $U$ which contains the set specified in the statement. 
\end{proof}
\begin{Rem}
 It is important to note that the above argument cannot relocate the potential (pole) singularity of $\alpha(z)$ at $z=1$ outside $D({\mathbf{0}},{\mathbf{1}})\cup D({\mathbf{1}},{\mathbf{1}})$. 
 \end{Rem}
 \begin{Rem}\label{rem-mps}
 Let $\alpha(z)$ be a tame power series. Since $\ln(z)/(1-z)$ is holomorphic in $D(1,1)$ the function
 $$
(-1)^\nu\ln(z)^\nu \alpha(z)= \left(\frac{ \ln(z)}{1-z}\right)^\nu (z-1)^\nu \alpha(z)
 $$
has a  multi-power series expansion $B_\alpha$ around  $1$ (with the same associated $N$ and $\be$) that converges  absolutely and uniformly in a neighborhood of $(0,1]$. This expansion is obtained from the multiplcation of $C_\alpha(z)$ and the power series expansion of $\displaystyle \left(\frac{ \ln(z)}{1-z}\right)^\nu$ in $D(1,1)$. 
\end{Rem}

\section{Bernoulli operators}  \label{BO}

\subsection{} Henceforth $(a_n)_{n\geq 1}$ will denote a sequence such that  $\alpha(z)=\sum_{n=0}^\infty  a_{n+1} z^n$ is tame and $\nu\geq 0$ will denote the order of its pole at $z=1$. We also denote by 
$$
\alpha_p(z)=k_\nu (z-1)^{-\nu}+\cdots+k_1 (z-1)^{-1}, \quad k_\nu\neq 0, 
$$ 
 the principal part of the Laurent expansion of $\alpha(z)$ around $z=1$ and let $\alpha_h(z)=\alpha(z)-\alpha_p(z)$ be the holomorphic part. We denote $k_0=\alpha_h(1)$.

The function $(-1)^\nu\ln(z)^\nu \alpha(z)$ has a multi-power series expansion around $z=1$ that converges absolutely and uniformly in a neighborhood of  $(0,1]$. 
Let us denote such an expansion, for certain fixed $N\geq 1$ and $\be\in \Z^N_{>0}$, by 
$$
B_\alpha(z)=\sum_{\mathbf{i}\in \Z_{\geq 0}^N} c_{\mathbf{i}}(\lambda_{\be}(z)-\mathbf{1})^{\mathbf{i}}.
$$
\begin{Def}
The Bernoulli operator associated to $\alpha$ (or rather to $B_\alpha$) is the  difference operator
\begin{equation}\label{eq10}
\B_\alpha=\sum_{\mathbf{i}\in \Z_{\geq 0}^N} c_{\mathbf{i}}(\E_{\be}-\id)^{\mathbf{i}}=\sum_{\mathbf{i}\in \Z_{\geq 0}^N} c_{\mathbf{i}}\Del_{\be}^{\mathbf{i}}.
\end{equation}
\end{Def}

Although the operators $\Del_{e_i}$ act as  bounded operators in certain Hilbert spaces (for example, on $L^2(\Re_+)$ they have norm $2$), the operator $\B_\alpha$ does not acquire a natural domain through standard techniques (holomorphic functional calculus)  without stronger assumptions on the domain of convergence of the multi-power series $B_\alpha(z)$. Instead, we will consider $\B_\alpha$ as an operator with domain the vector space $\D(\B_\alpha)$  of  $\Co$-valued functions $f(t)$ on $\Re_+$  for which the series
\begin{equation}\label{series-def}
\sum_{\mathbf{i}\in \Z_{\geq 0}^N} c_{\mathbf{i}}\Del_{\be}^{\mathbf{i}}f(t)
\end{equation}
converges absolutely and locally uniformly in $t$. In particular, on this domain, $\B_\alpha$ preserves the local integrability and continuity of the argument. 
\begin{Rem}
Although, by definition, the operator $\B_\alpha$ depends in principle on the multi-power series $B_\alpha$ rather than on $\alpha$, we will see that on certain spaces of functions the action of $\B_\alpha$ depends only on $\alpha$.
\end{Rem}
\subsection{}
We also need to consider the following function defined for $u$ in a neighborhood of $[0,\infty)$
$$
\beta(u)=B_\alpha(e^{-u})=u^\nu \alpha(e^{-u}).
$$
We can represent $\beta(u)$ on $\Re_+$ using the the power series expansion for $u^\nu \alpha(e^{-u})$ in $z=e^{-u}$ or the multi-power series expansion $B_\alpha(e^{-u})$  in $z=e^{-u}$. We note that 
\begin{equation}\label{eq19}
\beta(u)=O(1),\quad u\to 0^+\quad \text{and} \quad  \beta(u)=O(u^\nu), \quad u\to +\infty.
\end{equation}
\begin{Rem}\label{rem-expansion}
The series $B_\alpha(z)$ converges to $(-1)^\nu\ln(z)^\nu\alpha(z)$ and for $z$ in a neighborhood of $(0,1]$, in particular, for $z$ in a neighborhood of $1$. Therefore, for $u\in \Co$ in a neighborhood of $0$ we have that the series $B_\alpha(e^{-u})$ (the formal composition of the series $B_\alpha(z)$ and  the Taylor series of $e^{-u}$ around $u=0$) and the Taylor series of $\beta(u)=u^\nu \alpha(e^{-u})$ around $u=0$ coincide by the uniqueness of the power series expansion of a holomorphic function around $u=0$. 
\end{Rem}
\begin{Def}\label{def-todd}
Let $\T_\alpha(\del)$ denote the infinite order operator defined by substituting $u=\del$ in the Taylor expansion of $\beta(u)$ around $u=0$. The domain of $\T_\alpha(\del)$ is considered to be the space of polynomial functions in one variable. We will call $\T_\alpha(\del)$ the Todd operator associated to $\alpha$. The infinite order operator $$\T_\alpha=\T_\alpha(-\del)$$ is defined similarly, by substituting $u=-\del$ in the Taylor expansion of $\beta(u)$ around $u=0$.
\end{Def}

\begin{Rem} The terminology is justified by the fact that the operator $\T(\del)$ associated to the geometric series, formally written as
$$
\T(\del)=\frac{\del}{1-e^{-\del}}
$$
is precisely the Todd operator of  \cite{KhPu}. The corresponding function $\beta(u)$,  among other places, appears in algebraic topology as the characteristic power series for the total Todd class of a line bundle \cite{HirTop}*{\S10.1}. The Riemann-Roch theorem relates the Euler-Poincar\' e characteristic with a certain integral defined by the Todd class and the Chern character.  One of the main results of \cite{KhPu} is a Riemann-Roch theorem in combinatorial geometry, where the integral is replaced by the action of a (multi-dimensional) Todd operator.
\end{Rem}

\begin{Not}
We denote by $\T$ the operator $\T(-\del)$ associated to the  geometric series. Specifically, it is the operator obtained by substituting  $u=\del$ in the Taylor expansion of $\displaystyle\frac{u}{e^{u}-1}$ around $u=0$. The coefficients in the expansion are $B_k/k!$ where $B_k$ are the Bernoulli numbers.
\end{Not}

\begin{Rem}\label{rem-T}
We can write
$$
\T_\alpha=\T_{\alpha_p}+(-1)^\nu \del^\nu \T_{\alpha_h}\quad \text{and}\quad \T_{\alpha_p}=(-1)^\nu \sum_{n=1}^\nu k_n \T^n \del^{\nu-n}.
$$
\end{Rem}

\begin{Prop}\label{rem-todd}
On polynomial functions we have 
$$
\T_\alpha=\B_\alpha.
$$
\end{Prop}
\begin{proof}
From Remark \ref{rem-expansion} we infer that $\T_\alpha(-\del)$ can be defined by using the formal composition of the power series $B_\alpha(z)$  and the Taylor series of $e^{u}$  around $u=0$. In particular, for any function in its domain, $\T_\alpha(-\del)$ can be computed as $B_\alpha(e^{\del})$. However, on polynomial functions $e^{\del}$ acts as the forward shift operator $\E$. We obtain that 
$$
\T_\alpha=\B_\alpha
$$
on polynomial functions.
\end{proof}

\subsection{}
The operator $\B_\alpha$ can be regarded as an extension of action of the operator $\T_\alpha$ on polynomial functions. In search for a canonical domain, this fact opens the possibility of  using functional calculus for the operator $\del$ instead of that for $\Del$ or $\E$ in defining the operator $\B_\alpha$. In this case, the relevant context is that of \cite{BadOpe} which applies to closed, unbounded operators with spectrum contained  in a vertical strip. The spectrum of the operator $\del$ (acting on smooth functions on the real line) is the imaginary axis. To use the functional calculus, $\beta(u)$ would have to be, among other hypotheses, holomorphic in a vertical strip containing the imaginary axis. Our hypothesis on $\alpha(z)$ only guarantees that  $\beta(u)$ is holomorphic in a neighborhood of $0$ and, if $\alpha(z)$ has a pole at $z=1$ then $\beta(u)$ will have poles at the non-zero points of $2\pi i \Z$. Therefore, the functional calculus is largely incompatible with our situation. 

A certain renormalization allows one to extend the applicability of functional calculus. Specifically,  in \cite{HPFun} the semigroup associated to a closed operator (the semigroup would correspond to $e^{\del}$ in this situation) is defined in the usual fashion as the functional calculus for $e^{z}$ except that Ces\` aro (C,1) summation  is used to renormalize the integral. In this case,  the resulting (semi-)group  is $e^{\del}=\E$.  Overall, this has the same effect as constructing $\T_\alpha$ as $(-1)^\nu \del^\nu\A_\alpha$, leading us back to using discrete operators.


\subsection{} Let us consider the following polynomials.

\begin{Def}\label{def-bpoly}
For $n\geq 0$ define the $n$-th Bernoulli polynomial associated to $\alpha$ as
$$
B_\alpha[n;t]=\T_\alpha(-d)(t^n)=\B_\alpha(t^n).
$$
\end{Def}

\begin{Rem}
The polynomials $B[n;t]=\T(t^n)$ associated to the geometric series are precisely the classical Bernoulli polynomials. For $a\in \Z_{>0}$, the polynomials $B^{(a)}[n;t]=\T^a(t^n)$ are the generalized Bernoulli polynomials; in this case $\T^a=\T_\alpha$ for $\alpha(z)=(1-z)^{-a}$. We note that the generalized Bernoulli polynomials are defined in the same fashion for any $a\in \Co$.
\end{Rem}

\begin{Rem}\label{rem-bgs}
An equivalent way to define the polynomials $B_\alpha[t;n]$ is through their exponential generating series
$$
\sum_{n=0}^{\infty} B_\alpha[n;t]\frac{u^n}{n!}=(-u)^\nu \alpha(e^u)e^{tu}.
$$

\end{Rem}

\begin{Prop}\label{prop-top}
$B_\alpha[n;t]$, $n\geq 0$, is a polynomial of degree $n$; its top degree coefficient is $(-1)^\nu k_\nu$. In particular, $B_\alpha[0;t]=(-1)^\nu k_\nu$. Furthermore, 
\begin{equation}
\del B_\alpha[n;t]=n B_\alpha[n-1; t], \quad n\geq 1.
\end{equation}

\end{Prop}
\begin{proof}
If $\nu>0$ the operator $\del^\nu \T_{\alpha_h}$ acting on polynomials drops the degree. On the other hand, the action of $\T$ preserves the top degree monomial of the argument. The claim is now clear 
from the definition of $B_\alpha[n;t]$. Similarly, if $\nu=0$ then $\alpha_p(z)=0$ and the $\T_\alpha-k_0\id$ action on polynomials drops the degree.
\end{proof}

\begin{Rem}\label{rem-Bder}
For $0\leq n\leq \nu-1$, we have $\T_{\alpha_h}\del^\nu (t^n)=0$ and therefore $B_\alpha[n;t]=B_{\alpha_p}[n;t]$. The polynomial of largest degree among these is $B_\alpha[\nu-1; t]$. Because $\del$ commutes with $\T_\alpha$ we obtain that 
\begin{equation}\label{eq15}
\del^{n-1} B_\alpha[\nu-1;t]=(\nu-1)^{\underline{n-1}} B_\alpha[\nu-n; t], \quad 1\leq n\leq \nu.
\end{equation}
\end{Rem}

\subsection{} Any polynomial of degree $\nu-1$ will appear as the $\nu-1$-th Bernoulli polynomial for some tame series. 
\begin{Prop}\label{prop-taylor}
Let $t_0\in \Re_+$ and $P[t]=p_\nu(t-t_0)^{\nu-1}+\cdots+p_2(t-t_0)+p_1$, $p_\nu\neq 0$ a polynomial with complex coefficients. Then, there exist a tame series $\alpha(z)$ such that $B_\alpha[\nu-1;t]= P[t]$. Furthermore, all such tame series $\alpha(z)$ have the same $\alpha_p(z)$ and  
\begin{equation}\label{eq13}
B_\alpha[\nu-n;t_0]= \frac{(n-1)!}{(\nu-1)^{\underline{n-1}}} p_n,\quad 1\leq n\leq \nu.
\end{equation}
\end{Prop}

\begin{proof}
Let us first remark that the set $\{T^n \del^{\nu-n} (t^{\nu-1})\}_{1\leq n\leq \nu}$ is a basis for $\Co$-vector space of polynomials of degree at most $\nu-1$. This is because the expansion in the monomial basis of the indicated set is a uni-triangular matrix. We conclude that there exist unique constants $\{q_n\}_{1\leq n\leq \nu}$ such that 
$$
\sum_{n=1}^\nu q_n T^n \del^{\nu-n}(t^{\nu-1})=(-1)^\nu P[t].
$$
Note that $q_\nu=(-1)^\nu p_\nu\neq 0$. Now, 
$$
\Upsilon(z):=q_\nu (z-1)^{-\nu}+\cdots+q_1 (z-1)^{-1} 
$$ 
represents a tame power series with the required property. From Remark \ref{rem-T} we obtain that in fact for any such $\alpha(z)$ we must have that $\alpha_p(z)=\Upsilon(z)$. The remaining claim follows from Remark \ref{rem-Bder}.
\end{proof}
\subsection{} In fact, we can recover the holomorphic function $\alpha(z)$ from the Bernoulli polynomials.
\begin{Prop}\label{prop-u1}
The polynomials $B_\alpha[n;t]$, $n\geq 0$ uniquely determine $\alpha(z)$.
\end{Prop}

\begin{proof}
Since $\alpha(z)$ represents a meromorphic function around $z=1$, it is enough to argue that the Laurent expansion of $\alpha(z)$ around $z=1$ can be recovered from the Bernoulli polynomials. From Proposition \ref{prop-taylor} we know that $\alpha_p(z)$ can be recovered from $B_\alpha[\nu-1;t]$. We will show that $\alpha_h(z)$ can be similarly recovered. 
Denote by
$$
\alpha_h(z)=\sum_{n\geq 0} \frac{\varphi_n}{n!} (z-1)^n
$$
the expansion of $\alpha_h(z)$ around $z=1$. From the Fa\`a di Bruno formula \eqref{eq14} we know that formally
$$
\T_{\alpha_h}=\alpha_h(e^{\del})=\sum_{n\geq 0} \frac{\psi_n}{n!}\del^n, \quad \text{where}\quad \psi_0=\varphi_0 ~\text{and}~ \psi_n=\sum_{k=1}^n \Sn{n}{k}\varphi_k, ~n\geq 1.
$$
From the explicit expression relating $\psi_n$ and $\varphi_n$, $n\geq 0$, it is clear that $\alpha_h(z)$ is uniquely determined by the sequence $\psi_n$, $n\geq 0$.

For any $N\geq 0$, we show that $\psi_n$, $0\leq n\leq N$, can be recovered from $B_\alpha[N+\nu;t]$. Indeed,  the set $$\{T^n \del^{\nu-n} (t^{N+\nu})\}_{1\leq n\leq \nu}\cup \{\del^{\nu+n} (t^{N+\nu})\}_{0\leq n\leq N}$$ is a basis for $\Co$-vector space of polynomials of degree at most $N+\nu$. (Again, the expansion into the monomial basis produces a triangular matrix). Therefore, there is a unique linear combination of the elements of this basis that produce the polynomial $B_\alpha[N+\nu;t]$. Hence, $\psi_n$, $n\geq 0$, are uniquely determined.
\end{proof}

\subsection{} Another relevant fact is the following.

\begin{Prop}\label{prop-u2}
Let $t_0\in \Re_+$. The Bernoulli polynomials $B_\alpha[n;t]$, $n\geq 0$, are determined by the sequence $B_\alpha[n;t_0]$, $n\geq 0$.
\end{Prop}

\begin{proof}
We show by induction on $n\neq 0$, the coefficients of the Taylor expansion of $B_\alpha[n;t]$ around $t=t_0$ are determined by  $B_\alpha[k;t_0]$, $0\leq k\leq n$. For $n=0$, the polynomial $B_\alpha[0;t]$ is constant so it is determined by the value at any point. Assume that $n> 0$. From Proposition \ref{prop-top} and the induction hypothesis, all the coefficients of the Taylor expansion of $B_\alpha[n;t]$ around $t=t_0$ are determined except for the constant term. But the constant term is precisely $B_\alpha[n;t_0]$.
\end{proof}
\begin{Thm}\label{t4}
Let $t_0\in \Re_+$. The holomorphic function $\alpha(z)$ is uniquely determined by the sequence $B_\alpha[n;t_0]$, $n\geq 0$.
\end{Thm}
\begin{proof}
Straightforward from Proposition \ref{prop-u1} and Proposition \ref{prop-u2}.
\end{proof}

\begin{Rem}\label{rem-unq} We note that the only hypothesis about $\alpha(z)$ that was used is the fact that it is meromorphic in a neighborhood of $z=1$. While the argument above shows that for any family of polynomials there exists at most one corresponding $\alpha(z)$ it does not address the question of existence. A closer examination of the proof shows that it also proves the existence of a formal Laurent expansion around $z=1$ that produces a family of polynomials, compatible in the sense that they satisfy the properties specified by Proposition \ref{prop-top}. However, the analytic properties of this expansion, such as the fact that it represents a holomorphic function in $D(0,1)$ are not guaranteed without further assumptions on the family of Bernoulli polynomials. 
\end{Rem}

\subsection{} As it turns out, the functions $t^s\in\D(\B_\alpha)$ not only for  $s\in \Z_{\geq 0}$ but for $s\in \Co$. We will show first that $\Imag(\L)$ (which does not contain polynomial functions) is part of $\D(\B_\alpha)$.
\begin{Rem}\label{rem1}
Let $\B_\alpha$ as in \eqref{eq10} such that the series $F(\mathbf{z})=\sum_{\mathbf{i}\in \Z_{\geq 0}^N} c_{\mathbf{i}}(\mathbf{z}-\mathbf{1})^{\mathbf{i}}$ converges in $D(\mathbf{1},\mathbf{1})$. For example, the powers series $\alpha(z)$ as in Proposition \ref{prop-ML} have this property.  Let $f(t)$ such that for some continuous $\Z$-periodic functions $c_i(t)$, $1\leq i\leq N$, with values in $D(0,1)$ we have
\[
\Del_{e_i} f(t)=c_i(t) f(t).
\]
Then, the series \eqref{series-def} is absolutely convergent and $f(t)\in \D(\B_\alpha)$.
\end{Rem}

\begin{Exp} An example of this type is the exponential function $$f(t)=a^t.$$ If $a\in (0,2)$ then under the assumption that the series $F(\mathbf{z})=\sum_{\mathbf{i}\in \Z_{\geq 0}^N} c_{\mathbf{i}}(\mathbf{z}-\mathbf{1})^{\mathbf{i}}$ converges in $D(\mathbf{1},\mathbf{1})$, we obtain
$$
\B_\alpha(a^t)= a^t F(\lambda_\be(a)).
$$
If $a\in (0,1)$ then the convergence of $F(\mathbf{z})$ is not necessary since the convergence of  $\B_\alpha(a^t)$ follows from the convergence of the multi-power series $B_\alpha(z)$. In this case
$$
\B_\alpha(a^t)= a^tB_\alpha(a)=a^t (-1)^\nu\ln(a)^\nu \alpha(a)=(-1)^\nu \del^\nu(a^t)\alpha(a).
$$
\end{Exp}

\subsection{} As for the operator $\A_\alpha$, the basic example discussed above allows us to show that  $\Imag{\L}\subseteq \D(\B_\alpha)$.

\begin{Prop}\label{p1}
Let $f=\L(\varphi)\in \Imag(\L)$. Then, 
\begin{enumerate}[label={\roman*)}]
\item $f\in \D(\B_\alpha)$;
\item $\B_\alpha f=\L(\beta(u)\varphi(u))$;
\item $\ln(\id+\Del)f=\del f=\L(-u\varphi(u))$.
\end{enumerate}
In particular, $\B_\alpha$ can be considered as a linear operator
$
\B_\alpha: \Imag(\L) \to \Imag(\L).
$
\end{Prop}

\begin{proof} We have $\Del_\be^{\mathbf{i}}f(t)=\L((\lambda_\be(e^{-u})-\mathbf{1})^{\mathbf{i}}\varphi(u))$. For $u\in \Re_+$ the series 
$$
\sum_{\mathbf{i}\in \Z_{\geq 0}^N} c_{\mathbf{i}}(\lambda_{\be}(e^{-u})-\mathbf{1})^{\mathbf{i}}\varphi(u)
$$
converges absolutely and converges to $\beta(u)\varphi(u)$ pointwise. Since $\beta(u)=O(1)$, $u\to 0+$ and $\beta(u)=O(u^\nu)$, $u\to +\infty$ we obtain that $\beta(u)\varphi(u)\in \D(\L)$. By the Dominated Convergence Theorem we obtain that the series $\B_\alpha f$ converges absolutely and converges locally uniformly to $\L(\beta(u)\varphi(u))$ proving parts i) and ii).  Part iii) follows along similar lines with the remark that under our hypothesis $f(t)$ is differentiable and $\del f=\L(-u\varphi(u))\in \Imag(\L)$. 
\end{proof}

\subsection{} The relationship between the operators $\A_\alpha$ and $\B_\alpha$ is the following.

\begin{Thm}\label{t1}
Let $f=\L(\varphi)\in \Imag(\L)$ and $\varphi(u)=o(u^{\gamma})$, $u\to 0^+$,  for some $\gamma>-1$. Then
 $$\B_\alpha f(t)=(-1)^\nu\sum_{n=0}^\infty a_{n+1}f^{(\nu)}(t+n)=(-1)^\nu\A_\alpha f^{(\nu)}(t).$$
\end{Thm}
\begin{proof}
Straightforward from Proposition \ref{p1} and Proposition \ref{p2} (note that  $f^{(\nu)}(t)= \L((-u)^{\nu}\varphi(u))$ and $u^{\nu}\varphi(u)=o(u^{\gamma+\nu}).$)
\end{proof}

\begin{Rem}
We may consider $f$ as a function of a complex variable in $\Co_{0+}$ and we can clearly define the action of the operators $\ln(\id+\Del)$ and $\B$ on such functions.  Then, the conclusions of Proposition \ref{p1} and Theorem \ref{t1} hold with the real derivative of $f$ replaced with the complex derivative of $f$.
\end{Rem}
\begin{Rem}
With the hypotheses of Theorem \ref{t1}, it is clear that the action of $\B_\alpha$ on $\Imag(\L)$ depends only on $\alpha(z)$ (as opposed to the multi-power series $B_\alpha$).
\end{Rem}
\section{Analytic continuation}\label{AC}

\subsection{} We will show that if a certain one (complex) parameter family of functions lies partially within $\Imag(\L)$ then the entire family is in the domain of $\D(\B_\alpha)$. For this purpose, the following notation will be used henceforth. Let $f=\L(\varphi)\in \Imag(\L)$ and $s\in \Co_{1^+}$. We have $\varphi(u)u^{s-1}\in \D(\L)$ and 
we denote by
$$
f_{-s} (t)=\L(\varphi(u)u^{s-1}/\Gamma(s))(t),
$$
The Laplace-Mellin transform of $\varphi(u)$. In particular, $f_{-1}(t)=f(t)$. Note that $\Gamma(s)f_{-s}(t)$ is the Mellin transform of $e^{-tu}\varphi(u)$.

\subsection{} We first show that $f_{-s}(t)$ and $\B_\alpha f_{-s}(t)$ are indeed holomorphic in $s\in \Co_{1^+}$.
\begin{Prop}\label{p3}
Let $f=\L(\varphi)\in \Imag(\L)$ and $s\in \Co_{1^+}$. Then  
\begin{enumerate}[label={\roman*)}]
\item
$f_{-s}\in \D(\B_\alpha)$ and $\B_\alpha f_{-s}(t)=(\B_\alpha f)_{-s}(t)$;
\item If $\varphi(u)=o(u^\gamma)$, $u\to 0^+$, for some $\gamma>-1$ then
$$
\B_\alpha f_{-s}(t)=(-1)^\nu\sum_{n=0}^\infty a_{n+1} f_{-s}^{(\nu)}(t+n)=(-1)^\nu\A_\alpha f_{-s}^{(\nu)}(t);
$$
\item $\B_\alpha f_{-s}(t)$ is holomorphic in $s\in \Co_{1^+}$.
\end{enumerate}
\end{Prop}
\begin{proof} The first claim follows directly from Proposition \ref{p1}. We have $\varphi(u)u^{s-1}=o(u^{\gamma+s-1})$, $u\to 0^+$. Theorem \ref{t1} then gives conclusion of part ii).

For $s\in \Co_{1+}$
$$\int_0^\infty e^{-tu}\beta(u)\varphi(u)u^{s-1}/\Gamma(s)du$$
converges locally uniformly in $s\in \Co_{1^+}$. Consequently, the series 
$$
\sum_{n\geq 0} \int_n^{n+1} e^{-tu}\beta(u))\varphi(u)u^{s-1}/\Gamma(s)du
$$
converges locally uniformly in $s\in \Co_{1^+}$ and the Weierstrass theorem allows us to differentiate term by term with respect to $s$. In conclusion,  $\B_\alpha f_{-s}$ is holomorphic in $s\in \Co_{1^+}$.
\end{proof}

\begin{Cor}\label{cor}
The function $f_{-s}(t)$ is holomorphic as a function of $s\in \Co_{1^+}$ and 
$$
\del f_{-s}(t)=-s f_{-s-1}(t).
$$
\end{Cor}
\begin{proof}
Apply Proposition \ref{p3} for $\alpha(z)=1$. The second claim is clear from the definition of $f_{-s}(t)$.
\end{proof}
\begin{Rem}
Recall that, as a function in the image of the Laplace transform, $f_{-s}(t)$ is also holomorphic in $t\in \Co_{0^+}$ and therefore $f_{-s}(t)$ can be considered as a holomorphic function of two complex variables.
\end{Rem}

\begin{Rem}\label{rem-hol}
An important particular case is that of the constant function $\varphi(u)=1$. In this case, $$f_{-s}(t)=t^{-s}=\L(u^{s-1}/\Gamma(s))\quad \text{for} \quad s\in \Co_{1^+}.$$  Proposition \ref{p3}iii) implies that $\B_\alpha(t^{-s})$ is holomorphic in $s\in \Co_{1^+}$. 
\end{Rem}
\subsection{} The following subspace of $\D(\L)$ consists of functions that are dominated by some increasing function in the domain of $\L$ 
\begin{equation}
\D^{\iota}(\L):=\{\varphi~|~ |\varphi|\leq \psi,~\text{for some increasing}~\psi\in \D(\L)\}.
\end{equation}
\begin{Exp}
Let $\varphi\in \D(\L)$ such that $\varphi$ is continuous, bounded in a neighborhood of $0$, and $\varphi(t)=o(t^\gamma)$, $t\to +\infty$, for some $\gamma>0$. Then, $\varphi\in \D(\L)$.
\end{Exp}
For a function $f(s,t): \Co\times \Re_+\to \Co$ we denote by $\Reg(f)$ the set of points $s\in \Co$ which are regular points for $f(\cdot, t)$ as a complex single-valued analytic function and any $t\in \Re_+$. We denote by $\Sing(f)$ the complement of $\Reg(f)$ inside $\Co$. We say that $f(s,t)$ has isolated singularities if $\Sing(f)$ is a discrete set.

The following space of functions is a main ingredient in Theorem \ref{t2m}.
\begin{Def}\label{def-H}
Let $\H$ denote the space of functions $f(s,t): \Co\times \Re_+\to \Co$ satisfying the following properties
\begin{itemize}
\item $f(s,t)$ has isolated singularities; 
\item for $s\in\Reg(f)$, $f(s,t)$ is differentiable in $t\in \Re_+$ ;
\item $\Reg(f)\subseteq \Reg(\del f)$;
\item $f(t):=f(-1,t)=\L(\varphi)$ for some $\varphi\in \D^{\iota}(\L)$;
\item $f(-s,t)=f_{-s}(t)$ for $(s,t)\in \Co_{1}\times \Re_+$.
\end{itemize}
\end{Def}
\begin{convention} When we discuss functions $f(s,t)\in \H$ we assume that $f(t)$ is the corresponding function in the context of  Definition \ref{def-H}. 
\end{convention}

\begin{Rem} Let $f(s,t)\in \H$.
An immediate consequence of the definition is that $\del f(s,t)$ has isolated singularities and 
$$
\{s\in \Co~|~\Real(s)\leq -1\}\subset  \Reg(f).
$$
\end{Rem}

\begin{Rem}\label{rem-deriv} Let $f(s,t)\in \H$. Then $\Reg(\del f)-1\subseteq \Reg(f)$. Indeed, from Corollary \ref{cor} we obtain
$$
\del f(s,t)=s f(s-1,t)\quad \text{for all}\ (s,t)\in \Reg(f)\times\Re_+,~s-1\in \Reg(f).
$$
If $s_0\in \Reg(f)$ then we obtain that $(\del f(s,t))/s$ is regular and equals $f(s-1,t)$ in a neighborhood of $s_0$. Therefore $s_0-1$ is a removable singularity for $f(s,t)$, unless $s_0=0$. Note that the possible exception, $s_0=0$, is ruled out by the fact that $-1\in \Reg(f)$. We obtain that
\begin{equation}
 \Reg(f)\subseteq \Reg(\del f)\subseteq  \Reg(f)+1.
\end{equation}
\end{Rem}

\begin{Rem}
Let $f(s,t)\in \H$. Then, $\del f(s,t)\in \H$. In particular, for $s\in \Reg(f)$, $f(s,\cdot)\in \mathscr{C}^\infty(\Re_+)$ and 
\begin{equation}
\Reg(f)\subseteq \Reg(\del^n f) \subseteq \Reg(f)+n, \quad n\geq 0.
\end{equation}
This implies that $\Sing(f)$ is stable under the action of $\Z_{\geq 0}$ by translation.
\end{Rem}

\subsection{}
Let $g(t)\in\mathscr{C}^\infty(\Re_+)$, $\mathbf{h}\in \Re^N_{>0}$, and $\mathbf{i}\in \Z^N_{\geq 0}$. Then,
\begin{equation}\label{eq18}
\Del_{\mathbf{h}}^{\mathbf{i}} g(t)=\int_{[\mathbf{0},\mathbf{h}]^{\mathbf{i}}} \E_{\mathbf{u}} \left(\del^{|\mathbf{i}|}g(t) \right)\mathbf{du}.
\end{equation}
\begin{Prop}\label{peq9}
For $f(s,t)\in \H$ and $-s\in \Reg(f)$, we have
\begin{equation}\label{eq9}
(-\Del_{\mathbf{h}})^{\mathbf{i}} f(-s,t)=\frac{\Gamma(s+|\mathbf{i}|)}{\Gamma(s)} \int_{[\mathbf{0},\mathbf{h}]^{\mathbf{i}}} \E_{\mathbf{u}} \left(f({-s-|\mathbf{i}|,t)} \right)\mathbf{du}. 
\end{equation}
In particular, for $n\in\Z_{\geq 0}\cap \Reg(f)$, $f(n,t)$ is polynomial in $t$ of degree at most $n$.
\end{Prop}
\begin{proof} 
The first claim follows directly from \eqref{eq18} and Remark \ref{rem-deriv}.  For $(-\Del)^{n+1}$ and $s=-n$, remark that  the right-hand side of \eqref{eq9} is zero. This implies that $f(n,t)$ is polynomial in $t$ of degree at most $n$.
\end{proof}
\subsection{}
Recall the constant $C_{s,s_0}$ is defined by \eqref{eq-C} for $s_0>0$, $s\in \Co_{s_0^{-}}\cap\Co_{0^+}$. For $s$ in $\Co$ define
$$
n(s)=\min\{n\geq 0~|~ s+n\in \Co_{1^+}\},
$$
and 
\begin{equation}
\widetilde{C}_{s,s_0}:=C_{s+n(s),s_0+n(s)}.
\end{equation}
Note that, by Remark \ref{rem-bd1},  $\widetilde{C}_{s,s_0}\geq C_{s+n,s_0+n}$ for any $n\geq 0$ such that $s+n\in \Co_{1^+}.$
The following basic estimate was inspired by \cite{HasEin}*{pg. 463-464}. 
\begin{Lm}\label{lemma-estimate}
Let $f(s,t)\in \H$, $f=\L(\varphi)$ for $\varphi\in \D^\iota(\L)$, $s_0>1$, and $s\in \Co_{s_0^-}$. Fix $\psi\in \D(\L)$, increasing, such that $\psi\geq |\varphi|$ and denote $g=\L(\psi)$.  Then, for $-s\in \Reg(f)$ and $\mathbf{i}$ such that $s+|\mathbf{i}|\in \Co_{1^+}$, we have
$$
|(-\boldsymbol{\Del}_{\mathbf{h}})^{\mathbf{i}}f(-s,t)|\leq \widetilde{C}_{s,s_0}|(-\Del_{\mathbf{h}})^{\mathbf{i}}g_{-s_0}(t)| \left|\frac{\Gamma(s_0)}{\Gamma(s)}\right| \left|\frac{\Gamma(s+|\mathbf{i}|)}{\Gamma(s_0+|\mathbf{i}|)}\right||\mathbf{i}|^{\Real(s_0-s)}\left(\frac{t}{|\mathbf{i}|}+|\mathbf{h}|_2\right)^{\Real(s_0-s)}.
$$
\end{Lm}
\begin{proof}  Since  $\Real(s_0-s)>0$ we have that  $$\E_\mathbf{u} t^{\Real(s_0-s)}\leq (t+\<\be,\mathbf{i}\>)^{\Real(s_0-s)}\quad \text{for} \quad \mathbf{u}\in [\mathbf{0},\be]^{\mathbf{i}}.$$ From  \eqref{eq9} we obtain 
$$
|(-\Del_{\mathbf{h}})^{\mathbf{i}}f(-s,t)|\leq \left|\frac{\Gamma(s+|\mathbf{i}|)}{\Gamma(s)}\right| \int_{[\mathbf{0},\mathbf{h}]^{\mathbf{i}}} \E_{\mathbf{u}} \left| f(-s-|\mathbf{i}|,t) \right|\mathbf{du}.
$$
Our hypothesis readily implies that $f(-s-|\mathbf{i}|,t)=f_{-s-|\mathbf{i}|}(t)$ and $|f_{-s-|\mathbf{i}|}(t)|\leq g_{-s-|\mathbf{i}|}(t)$. Furthermore, Lemma \ref{lm-bd1} applied to $\phi(u)$, $s+|\mathbf{i}|$, and $s_0+|\mathbf{i}|$ gives us
$$
g_{-s-|\mathbf{i}|}(t)\leq C_{s+|\mathbf{i}|,s_0+|\mathbf{i}|} t^{\Real(s_0-s)} g_{-s_0-|\mathbf{i}|}(t).
$$
Putting everything together we obtain
$$
|(-\Del_{\mathbf{h}})^{\mathbf{i}}f(-s,t)|\leq C_{s+|\mathbf{i}|,s_0+|\mathbf{i}|} |(-\Del_{\mathbf{h}})^{\mathbf{i}}g_{-s_0}(t)| \left|\frac{\Gamma(s_0)}{\Gamma(s)}\right| \left|\frac{\Gamma(s+|\mathbf{i}|)}{\Gamma(s_0+|\mathbf{i}|)}\right||\mathbf{i}|^{\Real(s_0-s)}\left(\frac{t+\<\mathbf{h},\mathbf{i}\>}{|\mathbf{i}|}\right)^{\Real(s_0-s)},
$$
which implies our claim.
\end{proof}

\begin{Prop}\label{prop-estimate}
Let $\B=\sum_{\mathbf{i}\in \Z^N_{\geq 0}}c_\mathbf{i}\Del_{h\be}^\mathbf{i}$, and $s_0>1$, $s\in\Co_{s_0^-}$. 
Let $f(s,t)\in \H$, $f=\L(\varphi)$ for $\varphi\in \D^\iota(\L)$. Fix $\psi\in \D(\L)$, increasing, such that $\psi\geq |\varphi|$ and denote $g=\L(\psi)$. 

If $g_{-s_0}(t)\in \D(\B)$,  then $f(-s,t)\in \D(\B)$ for $-s\in \Reg(f)$. Furthermore, for $s\in \Co_{s_0^-}\cap(-\Reg(f))$, $\B f(-s,t)$ is holomorphic .
\end{Prop}

\begin{proof}
We know that  the series 
$$
\sum_{\mathbf{i}\in \Z^N_{\geq 0}}c_\mathbf{i}\Del_{h\be}^\mathbf{i}g_{-s_0}(t)
$$
is absolutely convergent. 

For $|\mathbf{i}|>-\Real(s)+1$, $\Gamma(s+|\mathbf{i}|)$ is finite and, as $|\mathbf{i}|$ approaches infinity, $\left|\frac{\Gamma(s+|\mathbf{i}|)}{\Gamma(s_0+|\mathbf{i}|)}\right||\mathbf{i}|^{\Real(s_0-s)}$ converges to $1$. Therefore, for a fixed constant $K>1$, there exist $k=k_{s,t}>0$ which depends locally uniformly on $s$ and $t$ such that 
$$
\left|\frac{\Gamma(s+|\mathbf{i}|)}{\Gamma(s_0+|\mathbf{i}|)}\right||\mathbf{i}|^{\Real(s_0-s)}\left(\frac{t}{|\mathbf{i}|}+|\be|_2\right)^{\Real(s_0-s)}<K(1+|\be|_2)^{\Real(s_0-s)}, \quad \text{for}\quad |\mathbf{i}|>k.
$$
Lemma \ref{lemma-estimate} now implies that  
$$
|\Del_{h\be}^\mathbf{i}f({-s},t)|<K\widetilde{C}_{s,s_0}(1+|\be|_2)^{\Real(s_0-s)}\left|\frac{\Gamma(s_0)}{\Gamma(s)}\right||\Del_{h\be}^\mathbf{i}g_{-s_0}(t)|, \quad \text{for}\quad |\mathbf{i}|>k.
$$
Therefore, the series 
$$
\sum_{\mathbf{i}\in \Z^N_{\geq 0}}c_\mathbf{i}\Del_{h\be}^\mathbf{i}f({-s},t)
$$
is absolutely convergent. The convergence is locally uniform in $t\in \Re_+$ and  $s\in \Co_{s_0-}$ and the series of term-by-term derivatives with respect to $s$ is of the same form. The Weierstrass theorem allows us to differentiate term-by-term. In conclusion, $\B f({-s},t)$ is holomorphic in the given domain.
\end{proof}

\subsection{}

Our first main result is the following. See also  \cite{EveMer}*{Theorem 3.3.4} (included below as Theorem \ref{thm: eve}) for a related result in the theory of Dirichlet series.

\begin{Thm}\label{t2m}
The following hold
\begin{enumerate}[label={\roman*)}]
\item $\H\subset \D(\B_\alpha)$;
\item For $f(s,t)\in \H$ we have $\B_\alpha f(s,t)\in \H$ and  $\Reg(f)= \Reg(\B_\alpha f)$.
\end{enumerate}
In particular, $\B_\alpha$ can be considered as a linear operator
$
\B_\alpha: \H \to \H.
$
\end{Thm}
\begin{proof} Fix $f(s,t)\in \H$, $f=\L(\varphi)$ for $\varphi\in \D^\iota(\L)$. Fix $\psi\in \D(\L)$, increasing, such that $\psi\geq |\varphi|$ and denote $g=\L(\psi)$. 
By Proposition \ref{p3}, $\B_\alpha f(-s,t)$ is holomorphic in $s\in \Co_{1^+}$. 
Fix $s_0>2$. Again, by Proposition \ref{p3}, $g_{-s_0}(t)\in \D(\B_\alpha)$.  Proposition \ref{prop-estimate} now implies that $f(-s,t)\in \D(\B_\alpha)$ for $s\in \Co_{2^-}$ and $\B_\alpha f({-s},t)$ is holomorphic in $s\in \Co_{2^-}$ as long as $-s\in \Reg(f)$. 
We conclude that $f(-s,t)\in \D(\B_\alpha)$ for any $-s\in \Reg(f)$ and $\B_\alpha f({-s},t)$ is holomorphic in the same domain.

By Proposition \ref{p1}, $\B_\alpha f=\L(\beta(u)\varphi(u))$. Recall that (see \eqref{eq19}) that the function $\beta(u)$ is continuous, bounded in a neighborhood of $0$ and $ \beta(u)=O(u^\nu)$, $u\to +\infty$. Therefore, $\beta(u)\varphi(u)\in \D^\iota(\L)$. From Proposition \ref{p3} we know that for $s\in \Co_{1^+}$, $\B_\alpha f(-s,t)=(\B_\alpha f)_{-s}(t)$.

Because the convergence of the series $\B_\alpha f(s,t)$ is locally uniform with respect to $t$, and $f(s,t)$ is differentiable with respect to $t$ we obtain that $\B_\alpha f(s,t)$ is differentiable with respect to $t$ and $\del$ commutes with $\B_\alpha$. Thus,
$$
\del\B_\alpha f(-s,t)=\B_\alpha \del f(-s,t)=-s\B_\alpha f(-s-1,t),
$$
and $\del\B_\alpha f(s,t)$ is therefore holomorphic in $s\in \Reg(f)$. The conditions of Definition \ref{def-H} are satisfied and in conclusion $\B_\alpha f(s,t)$ is an element of $\H$.
\end{proof}

\begin{Cor}\label{cor-independence}
The action of $\B_\alpha$ on $\H$ depends only on $\alpha(z)$.
\end{Cor}
\begin{proof}
Indeed, the action of $\B_\alpha$ on $\Imag(\L)$ depends only on $\alpha(z)$ and $\B_\alpha f(-s,t)$ is completely determined by its restriction to $s\in \Co_{1^+}$ for which $f(-s,t)\in \Imag(\L)$.
\end{proof}

\begin{Cor} Let $f(s,t)\in \H$ and assume that  $f=\L(\varphi)$ with $\varphi(u)=o(u^\gamma)$, $u\to 0^+$, for some $\gamma>-1$. Then, $\B_\alpha f(-s,t)$ is the analytic continuation of 
$$
s^{\overline{\nu}}\Di_\alpha^f(s+\nu,t)=s^{\overline{\nu}}\A_\alpha f(-s-\nu,t)=\B_\alpha f(-s, t), \quad s\in \Co_{1^+}, 
$$
and $\B_\alpha f(n,t)=\T_\alpha(-d) f(n,t)$ if $n\in \Reg(f)$.
\end{Cor}
\begin{proof}
Straightforward from Theorem \ref{t1}, Theorem \ref{t2m}, Proposition \ref{p3}, Remark \ref{rem-deriv}, Proposition \ref{rem-todd}, and Proposition \ref{peq9}.
\end{proof}
\subsection{} Another consequence of Theorem \ref{t2m} is the following

\begin{Thm}\label{thm-module} The map $$\Tcal\times \H\to \H,\quad (\alpha(z), f(s,t))\mapsto \B_\alpha f(s,t)$$
gives a $\Tcal$-module structure on $\H$.
\end{Thm}
\begin{proof}
We only have to argue that, for $\alpha,\alpha^\prime\in \Tcal$ we have $\B_\alpha\B_{\alpha^\prime}=\B_{\alpha\alpha^\prime}$ as operators on $\H$. This follows from Corollary \ref{cor-independence} and the fact that $B_\alpha B_{\alpha^\prime}$ is a multi-power series expansion corresponding to $\alpha \alpha^\prime$.
\end{proof}

It would be interesting to acquire some basic understanding of the $\Tcal$-mod structure on $\H$. For example, the submodule $\Tcal t^s$ consists of a certain class of Dirichlet series; we look more closely at this particular situation in the next section.


\section{Dirichlet series}\label{dirichlet}

\subsection{} In this section we explore the consequences of Theorem \ref{t2m} for the particular case $f(s,t)=t^s\in \H$ or, equivalently, for $\varphi(u)=1$. The first thing to note is that $f(s,t)=t^s$ is entire.

\begin{Thm}\label{t2}
The following hold
\begin{enumerate}[label={\roman*)}]
\item $t^{s}\in \D(\B_\alpha)$ for all $s\in \Co$;
\item $\B_\alpha(t^{s})$ is holomorphic in $s\in \Co$;
\item For $s\in \Co_{0^+}$ we have
$$
\B_\alpha(t^{-s})=s^{\overline{\nu}}\A_\alpha(t^{-s-\nu})=s^{\overline{\nu}}\Di_\alpha(s+\nu,t).
$$
\end{enumerate}
\end{Thm}
\begin{proof} 
By Theorem \ref{t2m}, $\B_\alpha(t^{s})$ is holomorphic in $s\in \Co$. For part iii) apply Theorem \ref{t1}.
\end{proof}

\subsection{} Theorem \ref{t2} can be restated in terms of the Dirichlet series associated to $\alpha$. Let us first define the following concept.
\begin{Def} Let $P[t]\in \Co[t]$ be a polynomial of degree $n$. The argument $t_0\in \Co$ is said to be generic for $P[t]$ if $\del^i P[t_0]\neq 0$ for all $0\leq i\leq n-1$, and otherwise it is said to be special. If $\nu\geq 1$, the generic and special arguments for the Dirichlet series $\Di_\alpha(s,t)$ are defined as the generic and, respectively,  special arguments of the Bernoulli polynomial $B_\alpha[\nu-1,t]$. If $\nu=0$, all arguments are considered to be generic.
\end{Def}

\begin{Thm}\label{t3}
The series $\Di_\alpha(s,t)$ 
\begin{enumerate}[label={\roman*)}]
\item converges absolutely and is holomorphic in $s\in \Co_{\nu^+}$;
\item has meromorphic continuation to $s\in \Co$ with only simple poles.
\end{enumerate}
More precisely, for $t_0\in \Co$ let $(p_n)_{1\leq n\leq \nu}$, $p_\nu\neq 0$, denote the coefficients in the expansion
$$
B_\alpha[\nu-1;t]=p_\nu(t-t_0)^{\nu-1}+\cdots+p_2(t-t_0)+p_1. 
$$
Then,
\begin{enumerate}[resume, label={\roman*)}]
\item The set of poles of $\Di_\alpha(s,t_0)$ is $P_\alpha(t_0)=\{n ~|~1\leq n\leq \nu, ~~p_n\neq 0\}$;
\item  If $\nu\geq 1$, $\Di_\alpha(s,t_0)$ has a simple pole at $s=\nu$ with residue $(-1)^\nu k_\nu/(\nu-1)!$;
\item The residue of $\Di_\alpha(s,t_0)$ at the pole $s=n$ is \begin{equation}\Res(\Di_\alpha(s,t_0), n)=\frac{(-1)^{\nu}}{(\nu-1)!}(-1)^np_n.\end{equation}
\end{enumerate}
Furthermore, for $n\geq 0$ we have  
\begin{enumerate}[resume, label={\roman*)}]
\item \begin{equation}\Di_\alpha(-n,t)=(-1)^\nu \frac{n!}{(\nu+n)!} B_\alpha[\nu+n;t].\end{equation}
\end{enumerate}
\end{Thm}

\begin{proof}
From Theorem \ref{t2}, $\Di_\alpha(s,t)$ converges for $s\in \Co_{\nu^+}$ to a holomorphic function in $s$ and the function $\B_\alpha(t^{-s+\nu})$ is the holomorphic continuation of $(s-\nu)^{\overline{\nu}}\Di_\alpha(s,t)$  to $s\in \Co$. If $\nu=0$ then $\Di_\alpha(s,t)$ is holomorphic. If $\nu\geq 1$, it is clear that the potential poles of $\Di_\alpha(s,t)$ are simple and form a subset of $\{1,2,\dots,\nu\}$. More precisely, $1\leq n\leq \nu$ is a pole of  $\Di_\alpha(s,t_0)$ if and only if 
$$\B_\alpha(t^{-n+\nu})=B_\alpha[-n+\nu;t]=\frac{1}{(\nu-1)^{\underline{n-1}}}\del^{(n-1)}B_\alpha[\nu-1;t]$$ 
does not vanish at $t=t_0$. This, in turn, translates to $p_n\neq 0$. In such a case
\begin{align*}
\Res(\Di_\alpha(s,t_0), n)&=\lim_{t\to t_0}\lim_{s\to n}\frac{s-n}{(s-\nu)^{\overline{\nu}}}\B_\alpha(t^{-s+\nu})\\
&=\frac{(-1)^{\nu-n}}{(\nu-n)!(n-1)!}B_\alpha[\nu-n;t_0]\\
&=\frac{(-1)^{\nu-n}}{(\nu-1)!}p_n.
\end{align*}
In the last step we used \eqref{eq13}. Since $B_\alpha[0;t]=(-1)^\nu k_\mu$ we obtain that the residue at $s=\nu$ equals $(-1)^\nu k_\mu/(\nu-1)!$.
The last claim is clear from Theorem \ref{t2}iii).
\end{proof}
\begin{Cor}\label{cor-t3}
If $t_0$ is a generic  argument for $\Di_\alpha(s,t)$ then the singularities of $\Di_\alpha(s,t_0)$ are simple poles at $s=1,\dots,\nu$. If $t_0$ is a special argument, $1\leq n\leq \nu$ is a removable singularity of $\Di_\alpha(s,t_0)$ if and only if $t=t_0$ is a critical point of $B_\alpha[\nu-n; t]$.
\end{Cor}

\subsection{} An immediate consequence of Theorem \ref{t3} is that it allows one to construct Dirichlet series with any specified  finite set of positive integers as simple poles and any specified residues at those poles.

\begin{Thm}\label{t7}
Let $t_0\in \Re_+$, $P\subset \Z_{>0}$ a finite subset, and $\{r_n~|~n\in P\}\subset \Co^*$. There exist a tame series $\alpha(z)$ such that the singularities of the associated Dirichlet series $\Di_\alpha(s, t_0)$ has the following properties
\begin{enumerate}[label={\alph*)}]
\item the singularities of $\Di_\alpha(s, t_0)$ are precisely the simple poles at $s=n\in P$;
\item $\Res(\Di(s, t_0), n)=r_n$, $n\in P$.
\end{enumerate}
Furthermore, all such tame series $\alpha(z)$ have the same principal part at $z=1$.
\end{Thm}

\begin{proof}
If $P= \emptyset$ we can pick $\alpha(z)$ to be any tame series which is holomorphic around $z=1$. Assume that $P\neq \emptyset$.
Let $\nu=\max(P)\geq 1$. Define $r_n=0$ for $1\leq n\leq \nu$, $n\not\in P$. For $1\leq n\leq \nu$ let $p_n:=(\nu-1)!(-1)^{\nu+n}r_n$ and
$$
P[t]:=p_\nu(t-t_0)^{\nu-1}+\cdots+p_2(t-t_0)+p_1.
$$
Since $\nu\in P$ we have $r_\nu\neq 0$ so $P[t]$ is a polynomial of degree $\nu-1$. From Proposition \ref{prop-taylor} we know that there exists a tame series $\alpha(z)$ such that $B_\alpha[\nu-1;t]=P[t]$. Theorem \ref{t3} implies now that the singularities of $\Di_\alpha(s, t_0)$ are precisely simple poles at $s=n\in P$ with residues $r_n$, respectively. The uniqueness  of $\alpha_p(z)$ follows again from Proposition \ref{prop-taylor}.
\end{proof}
Even more, we can show that the Dirichlet series $\Di_\alpha(s,t)$ is uniquely determined by it poles, residues, and values at negative integers.

\begin{Thm}\label{t5}
Let $t_0\in \Re_+$, $P\subset \Z_{>0}$ a finite subset, $\{r_n~|~n\in P\}\subset \Co^*$, and $(v_n)_{n\geq 0}\subset \Co$. Then, there exist at most one tame series $\alpha(z)$ such that the associated Dirichlet series $\Di_\alpha(s, t_0)$ has the following properties
\begin{enumerate}[label={\alph*)}]
\item the singularities of $\Di_\alpha(s, t_0)$ are precisely the simple poles at $s=n\in P$;
\item $\Res(\Di(s, t_0), n)=r_n$, $n\in P$;
\item $\Di_\alpha(-n,t_0)=v_n$, $n\geq 0$.
\end{enumerate}
\end{Thm}

\begin{proof} Assume that there is a tame series $\alpha(z)$ with the properties required in the statement.
As before, $t_0\in \Re_+$, $P\subset \Z_{>0}$ a finite subset, $\{r_n~|~n\in P\}\subset \Co^*$, and $(v_n)_{n\geq 0}\subset \Co$ uniquely determine the sequence of values $B_\alpha[n;t_0]$, $n\geq 0$. By Theorem \ref{t4} the sequence of values uniquely determine the Laurent expansion of $\alpha(z)$ around $z=1$ and therefore $\alpha(z)$.
\end{proof}

\begin{Rem}\label{rem-ext}
As pointed out in Remark \ref{rem-unq}, the data in the statement of Theorem \ref{t5} determines a formal Laurent expansion around $z=1$. The existence of a Dirichlet series specified by the data in Theorem \ref{t5} boils down to the study of the analytic properties of this formal Laurent expansion, specifically, whether it is the Laurent expansion around $z=1$ of some tame $\alpha(z)$. The problem of finding necessary or sufficient conditions for this to happen is akin to the classical converse theorems in the theory of L-functions \cites{Ham1, Ham2, Ham3, HecUbe, Wei, KMP}.
\end{Rem}
\subsection{} The results showing that the poles and residues of a Dirichlet series only depend on information about its holomorphic behavior around $z=1$ suggest that the third condition in the definition of a tame power series is not necessary for the existence of holomorphic continuation and the location of the poles. In what follows we show that this is indeed the case. Let $\eta(z)$ be a power series around $z=0$ that represents a function of one complex variable with the following properties
\begin{itemize}
\item $\eta(z)$ is holomorphic in $D(0,1)$;
\item $\eta(z)$ has a pole at $z=1$ (of order $\nu\geq 0$).
\end{itemize}
We adopt the usual notation for the corresponding Dirichlet series $\Di_\eta(s,t)$, the operators $\A_\eta$, $\T_\eta$, and Bernoulli polynomials $B_\eta[n;t]=\T_\eta(t^n)$ as well as for the principal part $\eta_p$ and homomorphic part $\eta_h(z)$ of the Laurent expansion of $\eta(z)$ around $z=1$. Note that $\eta_h(z)$ represents a function that is holomorphic in  $D(0,1)$  and around $z=1$ while $\eta_p(z)$ is tame. The following statement is known to specialists (see e.g. \cite{EveMer}*{Theorem 3.3.4} or \S\ref{ow}). We include a proof for the reader's convenience. 

\begin{Prop}\label{prop-hol}
Assume that $\nu=0$. Then, the series $\Di_{\eta}(s,t)$
\begin{enumerate}[label={\roman*)}]
\item converges absolutely and is holomorphic in $s\in \Co_{1^+}$;
\item has holomorphic continuation to $s\in \Co$.
\end{enumerate}
Furthermore, for $n\geq 0$ we have  
\begin{enumerate}[resume, label={\roman*)}]
\item \begin{equation}\Di_\eta(-n,t)=B_\eta[n;t].\end{equation}
\end{enumerate}
\end{Prop}
\begin{proof}

From Proposition \ref{p2}, for $s\in \Co_{0^+}$, $t^{-s}\in \D(\A_\eta)$  and $\Di_\eta(s,t)=\A_\eta(t^{-s})=\L(\eta(e^{-u})u^{s-1}/\Gamma(s))$ converges absolutely and is holomorphic in $s\in \Co_{1^+}$.  Let $0<\eps$ strictly smaller than the radius of convergence of the power series of 
$$
\eta(e^{-u})=\sum_{n=0}^\infty {c_n} u^n
$$ 
around $u=0$. Then,
$$
\L(\eta(e^{-u})u^{s-1}/\Gamma(s))(t)=\int_0^\eps e^{-ut}\eta(e^{-u})u^{s-1}/\Gamma(s)du+\int_\eps^\infty  e^{-ut}\eta(e^{-u})u^{s-1}/\Gamma(s) du
$$
The second integral is convergent for all $s\in \Co$ and represents an entire function.
The first integral, after the usual argument that justifies the interchange of integration and summation, writes as
\begin{equation}\label{eq16}
\sum_{n\geq 0} \frac{c_n}{t^{s+n}} \frac{\gamma(s+n,t\eps)}{\Gamma(s)}=\eps^s  \sum_{n\geq 0} {c_n}{\eps^n} s^{\overline{n}} {\gamma^*(s+n,t\eps)}
\end{equation}
where $\gamma(s,z)$ is the lower incomplete Gamma function and 
$$
\gamma^*(s,z)=\frac{\gamma(s,z)}{z^s\Gamma(s)}=e^{-z}\sum_{k=0}^{\infty}\frac{z^k}{\Gamma(s+k+1)}. 
$$
We refer to \cite{ParInc} for the properties of the functions $\gamma(s,z)$ and $\gamma^*(s,z)$. The function $\gamma^*(s,z)$ is entire in $s$ for fixed $z$ and since the series \eqref{eq16} converges absolutely for $s\in \Co_{1^+}$ we obtain that the series converges 
 locally uniformly on compact subsets of $\Co$ and therefore it defines an entire function.
 
 For the proof of part iii) we follow the argument used in the proof of \cite{EveMer}*{Theorem 3.3.4} (see the discussion in \S\ref{ow}). Fix $n\in \Z_{\geq  0}$ and for  $s\in \Co_{1^+}$ write 
\begin{align*}
 \Di_\eta(s,t)&=\frac{1}{\Gamma(s)}\int_0^\infty e^{-tu}\eta(e^{-u})u^{s-1} du\\ &= \frac{1}{\Gamma(s)}\int_0^\infty e^{-tu}\sum_{k=0}^{n+1} {c_k} u^{k+s-1} du +\frac{1}{\Gamma(s)}\int_0^\infty e^{-tu}\left(\eta(e^{-u})-\sum_{k=0}^{n+1} {c_k} u^k\right)u^{s-1} du\\
 &=\sum_{k=0}^{n+1} {c_k} s^{\overline{k}}t^{-s-k} +\frac{1}{\Gamma(s)}\int_0^\infty e^{-tu}\left(\eta(e^{-u})-\sum_{k=0}^{n+1} {c_k} u^k\right)u^{s-1} du.
\end{align*}
 The first term is polynomial of degree $n+1$  in $s$  and the second term, which we denote by $R_n(s,t)$, is given by an integral that converges locally uniformly for $\Real(s)>-n-1$. Consequently, the series 
$$
\sum_{m\geq 0} \frac{1}{\Gamma(s)}\int_m^{m+1} e^{-tu}\left(\eta(e^{-u})-\sum_{k=0}^{n+1} {c_k} u^k\right)u^{s-1} du
$$
converges locally uniformly in $\Real(s)>-n-1$ and the Weierstrass theorem allows us to differentiate term by term with respect to $s$. In conclusion,  $R_n(s,t)$ is holomorphic for  $\Real(s)>-n-1$. Moreover, $R_n(-N,t)=0$ for  $0\leq  N\leq n$ because of the corresponding pole of the Gamma function. Therefore, 
$$
 \Di_\eta(s,t)=\sum_{k=0}^{n+1} {c_k} s^{\overline{k}}t^{-s-k} + R_n(s,t)
$$
holds for $\Real(s)>-n-1$. For $s=-n$ we obtain $$\Di_\eta(-n, t)=\sum_{k=0}^{n} {c_k} (-n)^{\overline{k}}t^{n-k}=\sum_{k=0}^{n} (-1)^k{c_k} n^{\underline{k}}t^{n-k}=\T_\eta(t^n)=B_\eta[n;t],$$
which completes the proof for part iii).
\end{proof}

\subsection{}
We are now ready to state an extension of Theorem \ref{t3} for the series $\eta(z)$.

\begin{Thm}\label{t6}
The series $\Di_{\eta}(s,t)$ has meromorphic continuation to $s\in \Co$ with only simple poles. More precisely, for $t_0\in \Co$ let $(p_n)_{1\leq n\leq \nu}$, $p_\nu\neq 0$, denote the coefficients in the expansion
$$
B_\eta[\nu-1;t]=p_\nu(t-t_0)^{\nu-1}+\cdots+p_2(t-t_0)+p_1. 
$$
Then,
\begin{enumerate}[resume, label={\roman*)}]
\item The set of poles of $\Di_\eta(s,t_0)$ is $P_\eta(t_0)=\{n ~|~1\leq n\leq \nu, ~~p_n\neq 0\}$;
\item  If $\nu\geq 1$, $\Di_\eta(s,t_0)$ has a simple pole at $s=\nu$ with residue $(-1)^\nu k_\nu/(\nu-1)!$;
\item The residue of $\Di_\eta(s,t_0)$ at the pole $s=n$ is \begin{equation}\Res(\Di_\eta(s,t_0), n)=\frac{(-1)^{\nu}}{(\nu-1)!}(-1)^np_n.\end{equation}
\end{enumerate}
Furthermore, for $n\geq 0$ we have  
\begin{enumerate}[resume, label={\roman*)}]
\item \begin{equation}\Di_\alpha(-n,t)=(-1)^\nu \frac{n!}{(\nu+n)!} B_\alpha[\nu+n;t].\end{equation}
\end{enumerate}
\end{Thm}

\begin{proof} We can write 
$$
\Di_\eta(s,t)=\Di_{\eta_p}(s,t)+\Di_{\eta_h}(s,t).
$$
By Proposition \ref{prop-hol}  $\Di_{\eta_h}(s,t)$ has holomorphic continuation to the entire plane. The power series $\eta_p(z)$ is tame and its meromorphic continuation follows from Theorem \ref{t3}. The location of the poles, and the corresponding residues for $\Di_\eta(s,t)$ are inherited from those of $\Di_{\eta_p}(s,t)$ and the remaining claims follow directly from Theorem \ref{t3} applied to $\eta_p(z)$.  For the proof of part iv) we use Proposition \ref{prop-hol} iii) for  $\Di_{\eta_h}(-n,t)$ and Theorem  \ref{t3} vi) for  $\Di_{\eta_p}(-n,t)$.
\end{proof}

\subsection{} \label{ow}

We state  \cite{EveMer}*{Theorem 3.3.4} adapted to the notation used in this article and briefly discuss its relationship with the previous results. Let $\eta(z)=\sum_{n\geq 0} a_{n+1}z^n $ and $\phi(z)=\sum_{n\geq 0}b_{n+1}z^n$ be power series convergent in $D(0,1)$. Assume that $\eta(z)$ is meromorphic in a neighborhood of $z=1$ and denote by $\nu$ the order of $\eta(z)$ at $z=1$, with the convention that $\nu>0$ if $z=1$ is a pole of order $\nu$, and $\nu<0$ if $z=1$ is a zero of order $-\nu$.  Let $\eta(e^{-z})=\sum_{k=-\nu}^\infty c_k z^k$ denote the expansion of $\eta(e^{-z})$ in a neighborhood of $z=0$. We adopt the usual notation for the associated operators and Dirichlet series. For a Dirichlet series we denote by $\sigma_c$ the abscissa of convergence, by $\sigma_a$ the abscissa of absolute convergence, and by $\sigma_m$ the abscissa of the boundary line of the maximal open right half-plane in which it has meromorphic continuation (both are elements of $\Re\cup\{\pm\infty\}$, and $\sigma_m\leq \sigma_a$). 

\begin{Thm}\label{thm: eve}
With the notation above, assume that $\sigma_a(\Di_\phi)<\infty$ and 
 \begin{equation}\label{eq24} \sum_{1\leq i\leq n} a_i=O(n^r),\quad n\to \infty,\quad  \text{for some } r\geq 0.\end{equation} Then,
\begin{enumerate}[label=\roman*)]
\item $\sigma_a(\Di_{\eta\phi})\leq \max(r,\sigma_a(\Di_\phi)+r)$ and $\sigma_m(\Di_{\eta\phi})\leq \sigma_m(\Di_\phi)+\nu$;
\item For each $K\in \Z_{\geq 0}$, 
\begin{equation}\label{eq23}
\Di_{\eta\phi}(s)=\sum_{k=-\nu}^{K} c_{k}s^{\overline{k}}\Di_\phi(s+k)+R_K(s),
\end{equation}
for $\Real(s)>\max(\sigma_c(\Di_\phi)-K-1, -K-1, \sigma_m(\Di_\phi)+\nu)$, with $R_K(s)$ a holomorphic function in the domain $\Real(s)>\max(\sigma_c(\Di_\phi)-K-1, -K-1)$. Furthermore, $R_K(-m)=0$ for $-m\in \Z_{\leq 0}$ in its domain.
\end{enumerate}
\end{Thm}

This result, applied to $\phi(z)=1$, implies the analytic continuation of $\Di_\eta(s,1)$ as in Proposition \ref{prop-hol} and Theorem \ref{t6}. The arguments for the analytic continuation are to some extent distinct and Proposition \ref{prop-hol} does not make use of the hypothesis \eqref{eq24}. The equation \eqref{eq23} allows for the identification of the poles, residues,  and values at negative integers in the form given in Theorem \ref{t6}. In fact, the proof of Proposition \ref{prop-hol} iii), which identifies the values at negative integers, emulates the corresponding argument in the proof of Theorem \ref{thm: eve}.
There are also connections between Theorem \ref{thm: eve} and Theorem \ref{t2m} that we hope to discuss elsewhere.

\section{Examples}\label{sec-exp}

\subsection{Riemann zeta function} For $\alpha(z)=1/(1-z)$ we have $\nu=1$. The meromorphic continuation of the Dirichlet series $\Di_\alpha(s,t)$ is the Hurwitz zeta function and that of $\Di_\alpha(s)=\zeta(s)$ is the Riemann zeta function. As seen from Remark \ref{rem-bgs}, the polynomials $B_\alpha[t;n]$ are the classical Bernoulli polynomials. In this case, Theorem \ref{t3} reduces to the well-know classical properties of the Hurwitz zeta function.

\subsection{Dirichlet eta function}  For $\alpha(z)=1/(1+z)$ we have $\nu=0$. The corresponding Dirichlet series are 
$$
\Di_\alpha(s,t)=\sum_{n=0}^{\infty} \frac{(-1)^{n}}{(t+n)^s}\quad \text{and}\quad \Di_\alpha(s)=\sum_{n=1}^{\infty} \frac{(-1)^{n-1}}{n^s},
$$
the latter being known as Dirichlet's eta series. They both extend to entire functions in $s$ and their properties are usually studied through their relationship with the Riemann and Hurwitz zeta functions, but all this can be also read from Theorem \ref{t3}. By Remark  \ref{rem-bgs}, the exponential generating function for the corresponding Bernoulli polynomials is 
$$
\sum_{n=0}^{\infty} B_\alpha[n;t]\frac{u^n}{n!}=\frac{e^{tu}}{e^u+1}, 
$$
thus identifying $2B_\alpha[n;t]$ with the classical Euler polynomial $E_n(t)$.

\subsection{Dirichlet L-functions} Let $\chi:\Z\to \Co$ be a Dirichlet character of modulus $k$ and let
$$
\alpha(z)=\frac{1}{1-z^k}\sum_{i=1}^{k}\chi(i)z^{i-1}.
$$
Then, 
$$\Di_\alpha(s,t)=\sum_{n=0}^{\infty} \frac{\chi(n+1)}{(n+t)^s}=L(s,\chi,t),$$
and $L(s,\chi)=L(s,\chi,1)$ is the classical Dirichlet L-series. If $\chi$ is a principal character then $$\chi(1)+\cdots+\chi(k)=\varphi(k)\neq 0$$ (Euler's $\varphi$-function) and therefore $\nu=1$. Otherwise, $\chi(1)+\cdots+\chi(k)= 0$ (by orthogonality of characters) and $\nu=0$. So, the meromorphic extension of $L(s,\chi,t)$ has, for $\chi$ principal, a simple pole at $s=1$  with residue equal to minus the residue of $\alpha(z)$ at $z=1$ (hence equal to $\varphi(k)/k$) and it is otherwise entire.  The statements in Theorem \ref{t3} and Remark \ref{rem-bgs} are classical in this case. For completeness, we record the exponential generating function for the corresponding Bernoulli polynomials 
$$
\sum_{n=0}^{\infty} B_\alpha[n;t]\frac{u^n}{n!}=\frac{u}{e^{ku}-1}\sum_{i=1}^{k}\chi(i)e^{(t-1+i)u}.
$$

\subsection{Lerch zeta function}\label{sec-lerch} For $a\in\Co$ such that $\Imag(a)\geq 0$, the following series was introduced and studied in \cite{LerNot}
 $$\zeta(s,a,t)=\sum_{n=0}^\infty \frac{e^{2\pi i n a}}{(n+t)^s}.$$
The series  $\zeta(s,a)=e^{2\pi i a}\zeta(s,a,1)$ for $a$ real is called the periodic zeta series. One important number-theoretical aspect is that one of the constructions of $p$-adic L-functions \cite{Mor} is using interpolation of Lerch zeta function values. 

The function $\alpha(z)=1/(1-e^{2\pi i a}z)$ is holomorphic in holomorphic in $D(0,1)$ if and only if $\Imag(a)\geq 0$. In this case, the corresponding Dirichlet series is $\Di_\alpha(s,t)=\zeta(s,a,t)$. If $a\in \Z$, then the meromorphic extension of $\zeta(s,a,t)$ is precisely the Hurwitz zeta function. If $a\not \in \Z$ we have $\nu=0$ and the Dirichlet series extends to an entire function.  In this case, the exponential generating function of the corresponding Bernoulli polynomials  is 
$$
\sum_{n=0}^{\infty} B_\alpha[n;t]\frac{u^n}{n!}=\frac{e^{tu}}{1-e^{2\pi i a}e^u}.
$$
Their relationship with the Apostol-Bernoulli polynomials \cites{ApoLer, BoyApo} $\beta_n(t,e^{2\pi i n a})$ is the following
$$
\beta_{n+1}(t,e^{2\pi i n a})=-(n+1)B_\alpha[n;t],\quad n\geq 0.
$$
Under the change of variable $e^{2\pi i a}=w$ the Lerch zeta function is known as the Lerch transcendent \cite{EMOT} and arises as the meromorphic continuation of the series  $$\Phi(s,w,t)=\sum_{n=0}^\infty \frac{w^n}{(n+t)^s}.$$  
The condition $\Imag(a)\geq 0$ is equivalent to $|w|\leq 1$. The corresponding generating series is  $\alpha(z)=1/(1-wz)$. Again, for $w=1$, one goes back to the case of the Hurwitz zeta function. For $|w|\leq 1$, $w\neq 1$, the series $\Di_\alpha(s,t)$ extends to an entire function whose values at negative integers are given as $\Di_\alpha(-n, t)=B_\alpha[n;t]$, $n\geq 0$ where
$$
\sum_{n=0}^{\infty} B_\alpha[n;t]\frac{u^n}{n!}=\frac{e^{tu}}{1-we^u}.
$$
The coefficients of the polynomials $B_\alpha[n;t]$ can be qualitatively described: $B_\alpha[0;t]=1/(1-w)$ and $B_\alpha[n;0]$, $n>0$, is a polynomial in $w$ with rational coefficients. This, together with Proposition \ref{prop-top} implies that the coefficients of $B_\alpha[n;t]$ are polynomials in $w$, except for the top degree coefficient which equals $1/(1-w)$.

The polylogarithm is defined as  $\Li_s(w)= w\Phi(s,w,1)$ and it arises as the meromorphic continuation of the series 
$$\Li_s(w)=\sum_{n=1}^\infty \frac{w^n}{n^s}.$$ 
The polylogarithm, whose history goes back to Euler,  features prominently in the theory of motives and algebraic cycles but it is also ubiquitous throughout several mathematical and physical contexts (see e.g. \cites{Car, Lew, Mil, Oes, Zag}).
It is important to mention that the Lerch zeta function has also a continuation (with possible branch point singularities) for $\Imag(a)<0$ (or, equivalently, $|w|>1$). See, for example, \cites{LL, LL3} (also for historical context and commentary on previous work) where the continuation is established with the help of the four term functional equation \cite{LerNot}. However, for $|w|>1$ the corresponding function is no longer locally  represented by a Dirichlet series (e.g. $\Li_1(w)=-\ln(1-w)$ is no longer represented by a Dirichlet series anywhere outside $|w|<1$).

\subsection{Barnes zeta function} Let $\nu\geq 1$ and $\ba\in \Re^\nu_{>0}$. The Barnes zeta function \cite{Bar} arises as the meromorphic continuation of the series 
$$
\zeta_\ba (s, t) = \sum_{\mathbf{m} \in \Z_{\geq 0}^\nu} \frac{1}{\left(t + \<\mathbf{m}, \ba\> \right)^s}.
$$
For $\ba\in \Z^\nu_{>0}$ the series is a Dirichlet series. Assume that $\ba\in \Z^\nu_{>0}$ and let
$$
\alpha(z)=\prod_{i=1}^\nu\frac{1}{1-z^{a_i}}=\sum_{n=0}^\infty p_\ba(n)z^n.
$$
The function $p_\ba(n)$ is usually called a (partial) partition function. With this notation, 
$$
\Di_\alpha(s,t)= \sum_{n=0}^\infty \frac{p_\ba(n)}{\left(t +n \right)^s}=\zeta_\ba (s, t) .$$
The function $\alpha(z)$ has a pole at $z=1$ or order $\nu$. By Theorem \ref{t3}, the location of the poles, the corresponding residues and the values at negative integers are determined by the Bernoulli polynomials which, in this case, are precisely the Barnes-Bernoulli polynomials \cite{Bar} 
$$
\sum_{n=0}^{\infty} B_\alpha[n;t]\frac{u^n}{n!}=\prod_{i=1}^\nu\frac{u}{e^{a_iu}-1}e^{tu}.
$$
The meromorphic continuation, information about the poles, their residues, as well as the values at negative integers are due to Barnes \cite{Bar}, but the results have been revisited and reproved many times (e.g. \cite{RuiBar}). It would be interesting to provide an explicit expression for the polynomial $B_\alpha[\nu-1;t]$. Recall that this polynomial contains all the information about the behavior of $\Di_\alpha(s,t)$ at the poles.

The general case $\ba\in \Re^\nu_{>0}$ can be treated similarly. More precisely, one has to allow $\be \in \Re^\nu_{>0}$ in the definition of a multi-power series.

\subsection{Ehrhart series}
An interesting example of geometric origin is the following. Let $P\subset \Re^d$ be a convex $d$-polytope. Let $a_{n}=|nP\cap \Z^d|$, $n\geq 1$. The  series 
$$
Ehr(z)=1+\sum_{n=1}^{\infty} a_n z^n
$$
 is called the Ehrhart series of $P$. The relationship with the series $\alpha(z)$ is the following
 $$
 Ehr(z)=1+z\alpha(z).
 $$
If $P$ is in addition rational, then $a_n$ is quasi-polynomial in $n$ and for some positive integer $p$ (the smallest positive integer $k$ such that $kP$ is an integral polytope) and some polynomial $g(z)$ of degree at most $p(d+1)$ and constant term $1$ we have
$$
Ehr(z)=\frac{g(z)}{(1-z^p)^{d+1}}.
$$
Therefore, $\alpha(z)$ is a rational function of the form 
$$
\alpha(z)=\frac{h(z)}{(1-z^p)^{d+1}}
$$
and consequently tame. The corresponding Dirichlet series 
$$
\Di_\alpha(s,t)=\sum_{n=0}^\infty \frac{|(n+1)P\cap \Z^d|}{(t+n)^s}\quad \text{and}\quad \Di_\alpha(s)=\sum_{n=1}^\infty \frac{|nP\cap \Z^d|}{n^s}
$$
are subject to Theorem \ref{t3}. It is interesting to remark that for Delzant polytopes the Khovanskii-Pukhlikov theorem \cite{KhPu} expresses $|P\cap \Z^d|$ in terms of the action of a $d$-variable version of the classical Todd operator on the volume polynomial of $P$.

\subsection{Central binomial sums} Let $a_n=\binom{2n}{n}^{-1}$, $n\geq 1$. The corresponding power series $\alpha(z)$ has radius of convergence $4$ so, by Theorem \ref{t3},  the associated Dirichlet series extends to an entire function. The Dirichlet series
$$
\Di_\alpha(s)=\sum_{n=1}^\infty \frac{1}{n^s\binom{2n}{n}}
$$
is referred to in the literature as a central binomial sum. There are very interesting relations between the values $\Di_\alpha(s)$ for $s$ a positive integer and those of the Riemann zeta at positive integers \cites{BB, BBK}. The values at negative integers do not seem to have been investigated yet. From \cite{BBbook}*{pg. 385} we have
$$
\sum_{n=1}^\infty  \frac{(2z)^{2n}}{\binom{2n}{n}}=\frac{z^2}{1-z^2}+\frac{z\arcsin(z)}{(1-z^2)^{3/2}}
$$
which implies that
$$
\alpha(e^u)= \frac{1}{4-e^u}+\frac{4\arcsin(e^{u/2}/2)}{e^{u/2}(4-e^u)^{3/2}}.
$$
The exponential generating series for the Bernoulli polynomials is therefore
$$
\sum_{n=0}^{\infty} B_\alpha[n;t]\frac{u^n}{n!}=\left( \frac{1}{4-e^u}+\frac{4\arcsin(e^{u/2}/2)}{e^{u/2}(4-e^u)^{3/2}}\right) e^{tu}.
$$
According to Theorem \ref{t3}vi) we have $\Di_\alpha(-n)=B_\alpha[n;1]$, $n\geq 0$.

\subsection{Zeta-Dirichlet series} We use this as a generic term for Dirichlet series whose coefficients involve the values of the Riemann zeta function at positive integers. We will only look at one particular situation. 
Let $B_n$ denote  the classical Bernoulli numbers.
Let $a_1=0$ and $\displaystyle a_{n}=-\frac{1}{2}\cdot\frac{(2\pi i)^{n}B_{n}}{n!}$, $n> 1$. By Euler's formula, $a_{n}=0$ for $n$ odd and $a_{n}=\zeta(n)$ for $n$ even. The generating series is
$$
\alpha(z)=-\frac{1}{2}\sum_{n=1}^{\infty} \frac{(-1)^n(2\pi)^{2n}B_{2n}}{(2n)!}z^{2n-1}=-\frac{1}{2}({\pi \cot(\pi z)-1/z}),
$$
which is meromorphic in $\Co$ with only simple poles at $z\in \Z^*$. Therefore, $\alpha(z)$ is tame and, by Theorem \ref{t3}, the associated Dirichlet series 
$$
\Di_\alpha(s,t)=\sum_{n=1}^\infty \frac{\zeta(2n)}{(t+2n-1)^s} \quad \text{and}\quad  \Di_\alpha(s)=\sum_{n=1}^\infty \frac{\zeta(2n)}{(2n)^s} 
$$
extend to  meromorphic functions with only a simple pole at $s=1$ with residue $1/2$. 
The exponential generating series for the Bernoulli polynomials is therefore
$$
\sum_{n=0}^{\infty} B_\alpha[n;t]\frac{u^n}{n!}=\left( \frac{u(\pi e^u\cot(\pi e^u)-1)}{2e^u}\right) e^{tu}.
$$
According to Theorem \ref{t3}vi) we have $\Di_\alpha(-n, t)=B_\alpha[n;t]$, $n\geq 0$.

\subsection{Jacobi polynomials} The Jacobi polynomials  $P_n^{(a,b)}(x)$ are a family of polynomials on the interval $(-1,1)$, orthogonal with respect to the weight function $(1-x)^a(1+x)^b$. They sit high in the hierarchy of classical special orthogonal polynomials, only below the Wilson polynomials. It is arguably possible to obtain results similar to those in this section for Wilson polynomials; we restrict to Jacobi polynomials for technical convenience.

For us, the parameters $a,b\in \Co$ and $x \in (-1,1)$. The generating function for Jacobi polynomials   is given by \cite{AAR}*{Theorem 6.4.2}
$$
\alpha(z)=\sum_{n=0}^{\infty} P_n^{(a,b)}(x)z^n=2^{a+b}R^{-1}(R+1-z)^{-a}(R+1+z)^{-b}, \quad \text{where}\ R=R(x,z)=(z^2-2xz+1)^{1/2}.
$$
We work with the principal branch of the complex power function; the branch cut is made along the negative real axis. The power series has radius of convergence $1$. A routine computation shows that the only singularities of the analytic function represented by $\alpha(z)$ are branch points at $z$ for which $|z|^2=1,~\Real(z)=x$, and $|z|^2=2,~\Real(z)=x\pm 1/2$. Given that $x\in (-1,1)$, $\alpha(z)$ is holomorphic in a neighborhood of $z=1$ and Theorem \ref{t6} implies that the corresponding Dirichlet series
$$
\Di_\alpha(s,t)=\sum_{n=0}^\infty \frac{P_{n}^{(a,b)}(x)}{(t+n)^s} \quad \text{and}\quad  \Di_\alpha(s)=\sum_{n=1}^\infty \frac{P_{n-1}^{(a,b)}(x)}{n^s} 
$$
have holomorphic continuation to $s\in \Co$. If $x\in(-1,1/2]$ then $\alpha(z)$ satisfies the hypotheses of Proposition \ref{prop-ML} and it is therefore tame. In this case, Theorem \ref{t3} allows the computation of the values at negative integers of this Dirichlet series in terms of the associated Bernoulli polynomials. Their exponential generating series is 
$$
\sum_{n=0}^{\infty} B_\alpha[n;t]\frac{u^n}{n!}=2^{a+b}\widetilde{R}^{-1}(\widetilde{R}+1-z)^{-a}(\widetilde{R}+1+z)^{-b} e^{tu}, \quad \text{where}\ \widetilde{R}=\widetilde{R}(x,u)=(e^{2u}-2xe^u+1)^{1/2}.
$$
It would be interesting to investigate whether $\alpha(z)$ remains tame also for $x\in (1/2,1)$.

\section{Other singularities}\label{sec-other}

\subsection{} We briefly discuss some known examples of series $\alpha(z)$ for which the singular point $z=1$ is not a pole and thus fall outside the hypotheses of our main results. In these examples, $z=1$ is either a branch point or a non-isolated singularity. However, in these examples the Dirichlet series $\Di_\alpha(s)$ (and therefore $\Di_\alpha(s,t)$) admits meromorphic continuation to $s\in \Co$. At present there is no general technique that produces such a meromorphic continuation;  it is usually obtained either by directly relating $\Di_\alpha(s)$ to a Dirichlet series of tame type or, especially in arithmetic situations, by making use of functional equations.

\subsection{Branch points} Recall from \S\ref{sec-lerch} that, for fixed $s_0\in \Co$, the series $$\Phi(s_0,z,1)=\sum_{n=0}^\infty \frac{z^n}{(n+1)^{s_0}}=\frac{Li_{s_0}(z)}{z}$$
has radius of convergence $1$ and is also convergent for $|z|=1$ if $s_0\in \Co_{1^+}$. The series extends to a holomorphic  function in $z\in \Co$ with at most a singularity at $z=1$ \cites{GS, LL3}. More precisely,  for $s_0 \in \Z_{\leq 0}$, $\Phi(s_0,z,1)$ is a rational function in $z$ with a simple pole at $z=1$ (see \S\ref{sec-lerch} or \cite{LL3}*{Theorem 2.5}). For all other values of $s_0$,  $z=1$ is a branch point so its maximal domain of holomorphy is not included in the complex plane. However, one can consider a  maximal domain of holomorphy inside the complex plane by making a branch cut (e.g. along $[1,\infty)$). The monodromy around $z=1$ is described in  \cite{LL3}*{Theorem 2.4}.

For $\alpha(z)=\Phi(s_0,z,1)$, the corresponding Dirichlet series
$$
\Di_\alpha(s)=\sum_{n=1}^{\infty} \frac{n^{-s_0}}{n^s}=\zeta(s+s_0)
$$
has meromorphic extension to $s\in \Co$ with a simple pole at $s=1-s_0$. It is interesting to note that the location of the pole falls within $\Z_{>0}$ (as in Theorems \ref{t3}, \ref{t7}, \ref{t6}) precisely when the singularity of $\alpha(z)$ at  $z=1$ is a pole. When the pole of $\Di_\alpha(s)$ falls in $\Z_{\leq 0}$, there is special monodromy around the branch point $z=1$. Overall, it is not yet clear whether the location of the poles of $\Di_\alpha(s)$ can be read from the monodromy around $z=1$, or other invariants have to be considered.

Another example of the same nature that should play a role in evaluating this phenomenon is the following. Consider $\alpha(z)=-\sum_{n=0}^{\infty}\ln(n+1)z^n$. It can be shown that $z=1$ is a branch point and $\alpha(z)$ has no other singularities in $D(0,2\pi)$. The corresponding Dirichlet series
$$
\Di_\alpha(s)=-\sum_{n=1}^{\infty} \frac{\ln(n)}{n^s}=\zeta^\prime(s)
$$
has meromorphic extension to $s\in \Co$ with a double pole at $s=1$. In this case, the location of the pole falls within $\Z_{>0}$ but the pole is no longer simple.

\subsection{Circle cuts} Another important situation is when $z=1$ is not an isolated singularity for $\alpha(z)$. A particular case consists of power series for which the unit circle is the natural boundary.
The following result, due to Fritz Carlson \cite{Carl} (see also \cite{Pol}), will be useful.
\begin{Thm}\label{t8}
Let $\alpha(z)$ be a power series with radius of convergence $1$ and integer coefficients. Then either $\alpha(z)$  has the unit circle as its natural boundary, or $\alpha(z)$ can be represented as a rational function of the form
$$
\frac{P(z)}{(1-z^p)^q},
$$
for some $P(z)\in \Z[z]$ and $p,q\in \Z_{>0}$.
\end{Thm}
In the discussion below, we will largely restrict to the case of power series with integral coefficients and therefore will be concerned with power series with the unit circle as a cut. If $\alpha_1(z)$ is of this type and $\alpha_2(z)$ has $z=1$ as an isolated singularity then $\alpha_1(z)+\alpha_2(z)$ has again the unit circle as a cut. Therefore, the question of understanding when does $\Di_\alpha(s)$ have meromorphic continuation to $s\in \Co$  hinges in part on the ability to characterize  $\alpha(z)$, with the unit circle as its natural boundary, for which $\Di_\alpha(z)$ is entire.

\subsubsection{Dedekind zeta function} Let $k$ be an algebraic number field of degree $N$ and $$\zeta_k(s)=\sum_{n=1}^\infty \frac{a_n}{n^s}$$ 
the corresponding Dedekind zeta function; the coefficient $a_n$ counts the number of ideals of norm $n$ in the  integer ring of $k$. The series converges for $\Real(s)>1$. Its meromorphic continuation to $s\in \Co$ is due to Hecke (see e.g. \cite{Hec}) and it is based on the existence of a functional equation. The resulting  function has only a simple pole at $s=1$. The residue is computed by the analytic class number formula (in this form due to Dedekind)
$$
\Res(\zeta_k(s), 1)=\frac{2^{r_1}(2\pi)^{r_2}h_k R_k}{w_k\sqrt{|D_k|}}
$$
where $r_1$, $2r_2$ denote the number of real embeddings and complex embeddings of $k$, $h_k$ the class number, $w_k$ the number of roots of unity in $k$, $D_k$ the discriminant, and $R_k$ the regulator. The regulator, in particular, is typically a transcendental real number, defined as a determinant of logarithms of algebraic numbers (units in $O_k$). 

The power series $\alpha(z)=\sum_{n\geq 0}a_{n+1}z^n$ has radius of convergence at least $1$ (since the Dirichlet series has non-empty domain of convergence) and it cannot be holomorphic in a neighborhood of $z=1$ (otherwise the Dirichlet series would be entire). Therefore, the radius of convergence is precisely $1$ and, according to Theorem \ref{t8}, $\alpha(z)$ is either a rational function or has the circle as a cut. If  $\alpha(z)$ is a rational function then, again by Theorem \ref{t8} and Theorem  \ref{t6}, the pole at $z=1$ is simple and the residue of $\zeta_k(s)$ at $s=1$ is a rational number. Therefore, if 
$$\frac{\pi^{r_2}R_k}{\sqrt{|D_k|}}\not\in \Rat$$
(which seems to be always the case if $n>1$), then $\alpha(z)$ has the circle as its natural boundary. To wit, the Dirichlet series has pole that falls in $\Z_{>0}$ but the residue is not what it is expected if it would arise from a rational function.

\subsubsection{Modular forms}  To the sequence $(a_n)_{n\geq 1}$ and $\lambda>0$ we associate the function
$$
f(\tau)=\sum_{n=1}^\infty a_n e^{2\pi i n \tau/\lambda}.
$$
It is clear that $\tau$ belongs to the upper half-plane if and only if  $e^{2\pi i \tau/\lambda}\in D(0,1)$.
The relationship between $f(z)$ and $\alpha(z)$ is the following
$$
f(\tau)=e^{2\pi i  \tau/\lambda}\alpha(e^{2\pi i  \tau/\lambda}).
$$
In particular, $f(\tau)$  is analytic inthe upper half-plane. We assume that $a_n=O(n^c)$, $n\to +\infty$, for some $c\in \Re$. If, in addition, $f(\tau)$ satisfies 
\begin{equation}\label{eq20}
f(-1/\tau)=\gamma(\tau/i)^k f(\tau),
\end{equation}
for some $k>0$ and $\gamma=\pm1$, we say that $f(\tau)$ is a cusp form of weight $k$ and multiplier $\gamma$ and write $f(\tau)\in M_0(\lambda, k, \gamma)$. Hecke \cite{HecUbe} (see \cite{BK} for a detailed account) developed a correspondence between modular forms (the relevant group $G(\lambda)$ being the group generated by the maps $\tau\mapsto -1/\tau$ and $\tau\mapsto \tau +\lambda$) and Dirichlet series. For our situation, the correspondence 
\cite{BK}*{Theorem 2.1} says that \eqref{eq20} is satisfied if and only if $\Phi(s)=(2\pi/\lambda)^{-s}\Gamma(s)\Di_\alpha(s)$ has holomorphic extension to an entire function bounded in each vertical strip and
\begin{equation}
\Phi(s)=\gamma\Phi(k-s).
\end{equation}

Assume that $f(\tau)\in M_0(\lambda, k, \gamma)$.
For $\lambda>2$ the vector space $M_0(\lambda, k, \gamma)$ is infinite dimensional and the standard fundamental region for $G(\lambda)$ has on its boundary the line segments $(-\lambda/2,-1)$ and $(1,\lambda/2)$. Under favorable conditions one can use of the Schwarz reflection principle to extend $f(\tau)$ to a holomorphic function into the lower half-plane, with potential singularities (see \cite{BK}*{Chapter 4} for details). In this case, it is possible that $z=1$ is a regular point for $\alpha(z)$ and it is covered by Theorem \ref{t6}.

On the other hand, for $\lambda< 2$  the vector space $M_0(\lambda, k, \gamma)$ is finite dimensional. The real line is the natural boundary for $f(\tau)$ and therefore the unit circle is the natural boundary for $\alpha(z)$. Indeed, if $f(\tau)$ extends to a neighborhood of a point on the real line then, because any such neighborhood contains a fundamental domain for $G(\lambda)$, $f(\tau)$ extends to a bounded entire function and it is therefore the constant function $0$.  This case thus provides examples of entire Dirichlet series that do not arise from tame power series, as well as elements of $\H$ that are not in $\Tcal t^s$. Nevertheless, any such Dirichlet series will generate  a $\Tcal$-submodule of $\H$.

\subsubsection{Arithmetic convolution} The arithmetic convolution $(\alpha_1*\alpha_2)(z)$ of the power series $\alpha_1(z)$, $\alpha_2(z)$ is the unique series for which $\Di_{\alpha_1*\alpha_2}(s,t)=\Di_{\alpha_1}(s,t)\Di_{\alpha_2}(s,t)$. The following result provides another class of examples of series with the unit circle as a natural boundary for which the associated Dirichlet series admit meromorphic continuation to the complex plane.
\begin{Thm}\label{t9}
Let $\alpha_1(z)$, $\alpha_2(z)$ be power series with non-negative integer coefficients such that their radius of convergence equals $1$ and they represent rational functions. Then,
\begin{enumerate}[label={\roman*)}]
\item $(\alpha_1*\alpha_2)(z)$ has the unit circle as its natural boundary;
\item $\Di_{\alpha_1*\alpha_2}(s,t)$ has meromorphic extension to $s\in\Co$.
\end{enumerate} 
\end{Thm}
\begin{proof}
The series $\alpha_1(z)$, $\alpha_2(z)$ must have $z=1$ as a pole. By Theorem \ref{t3} $\Di_{\alpha_1}(s,t)$ and $\Di_{\alpha_2}(s,t)$ have meromorphic continuation to $s\in \Co$ with only simple poles implying that $\Di_{\alpha_1*\alpha_2}(s,t)$ has meromorphic continuation to $s\in \Co$ with at most double poles. In fact, for $t_0$ that is generic for both $\alpha_1(z)$ and $\alpha_2(z)$, both $\Di_{\alpha_1}(s,t)$ and $\Di_{\alpha_2}(s,t)$ have a pole at $s=1$ and therefore $\Di_{\alpha_1*\alpha_2}(s,t_0)$ has a double pole at $s=1$. The series $(\alpha_1*\alpha_2)(z)$ must have radius of convergence at least $1$ (because its Dirichlet series converges in some right half-plane) and cannot be holomorphic in a neighborhood of $z=1$ (otherwise Theorem \ref{t6} forces $\Di_{\alpha_1*\alpha_2}(s,t)$ to be entire). Furthermore, $(\alpha_1*\alpha_2)(z)$ cannot represent a rational function because 
 $\Di_{\alpha_1*\alpha_2}(s,t_0)$ has a double pole at $s=1$. Theorem \ref{t8} now implies that $(\alpha_1*\alpha_2)(z)$ has the unit circle as its natural boundary.
\end{proof}
The same principle applies to the series corresponding to the quotient $\Di_{\alpha_1}(s,t)/\Di_{\alpha_2}(s,t)$ if we have some information about the zeroes of $\Di_{\alpha_2}(s,t)$ (at least for some value of $t$). Results of this type imply that many of the Dirichlet series constructed from arithmetical sequences (Moebius function, Euler $\varphi$-function, divisor functions, etc.) have the unit circle as a cut. If also shows that Dirichlet series $\Di_\alpha(s)$ associated to $\alpha(z)$ with the circle as a cut  do acquire poles that correspond to zeroes of other Dirichlet series. In some instances, these zeroes are expected to lie on vertical lines, hence the poles of $\Di_\alpha(s)$ are expected to lie on vertical lines. There are examples in the literature for which this known to be the case \cites{AMP, Eve}.

One class of Dirichlet series that was extensively studied is the extended Selberg class $S$. Such a Dirichlet series is required to converge for $s\in \Co_{1^+}$, and have meromorphic extension to the complex plane with only a possible pole at $z=1$, satisfy a functional equation of Riemann type, have an Euler product, and satisfy the Ramanujan hypothesis. We refer to \cite{Kac} for a survey of the theory and related developments. An element $F\in S$ is primitive if $F=F_1F_2$, $F_1,F_2\in S$ implies $F_1=1$ or $F_2=1$. As it turns out \cite{CG}, every element of $S$ can be factored into primitive factors and, assuming the Selberg  Orthogonality Conjecture, the factorization is unique up to the order of the factors; furthermore, the Riemann zeta function is the only polar primitive element of $S$. This is consistent,  in particular, with the Dedekind conjecture on the Dedekind zeta function and with the Artin conjecture. Therefore, up to arithmetic convolution, understanding the structure of $S$ depends, as remarked before, on the ability to characterize the power functions with $z=1$ a non-isolated singularity for which Dirichlet series is entire.

\subsection{Non-isolated singularities} Historically, the first examples of series that have the circle of their disk of convergence as the natural boundary were lacunary series. We only point out that in general these might produce Dirichlet series that would fall outside the range of what might be considered interesting. One example is that of the series defined as  $\displaystyle z\alpha(z)=\sum_{p~\text{prime}} z^p$. A classical theorem of Mandelbrojt \cite{Man}*{Theorem 4} applies to show that the series has radius of convergence $1$ and its set of singularities on the unit circle is irreducible. In particular, $z=1$ is a non-isolated singularity. The corresponding Dirichlet series
$$
\Di_\alpha(s)=\sum_{p~\text{prime}} \frac{1}{p^s},
$$
was studied by Landau and Walfisz \cite{LW}. As it turns out, $\Di_\alpha(s)$ has $\Real(s)=0$ as natural boundary.




\begin{bibdiv}
\begin{biblist}[\normalsize]
\BibSpec{article}{%
+{}{\PrintAuthors} {author}
+{,}{ }{title}
+{.}{ \textit}{journal}
+{}{ \textbf} {volume}
+{}{ \PrintDatePV}{date}
+{,}{ no. }{number}
+{,}{ }{pages}
+{,}{ }{status}
+{.}{}{transition}
}

\BibSpec{book}{%
+{}{\PrintAuthors} {author}
+{,}{ \textit}{title}
+{.}{ }{series}
+{,}{ vol. } {volume}
+{,}{ \PrintEdition} {edition}
+{,}{ }{publisher}
+{,}{ }{place}
+{,}{ }{date}
+{,}{ }{status}
+{.}{}{transition}
}

\BibSpec{collection.article}{
+{}{\PrintAuthors} {author}
+{,}{ \textit}{title}
+{.}{ In: \textit}{conference}
+{,}{ }{pages}
+{.}{ }{series}
+{,}{ vol. } {volume}
+{,}{ }{publisher}
+{,}{ }{place}
+{,}{ }{date}
+{,}{ }{status}
+{.}{}{transition}
}

\BibSpec{collection}{
+{}{\PrintAuthors}{editor}
+{,}{ \textit}{title}
+{.}{ In: \textit}{conference}
+{,}{ }{pages}
+{.}{ }{series}
+{,}{ vol. } {volume}
+{,}{ }{publisher}
+{,}{ }{place}
+{,}{ }{date}
+{,}{ }{status}
+{.}{}{transition}
}


\bib{AMP}{article}{
   author={Allouche, J.-P.},
   author={Mend\`es France, M.},
   author={Peyri\`ere, J.},
   title={Automatic Dirichlet series},
   journal={J. Number Theory},
   volume={81},
   date={2000},
   number={2},
   pages={359--373},
   issn={0022-314X},
   review={\MR{1752260}},
   doi={10.1006/jnth.1999.2487},
}


\bib{AAR}{book}{
   author={Andrews, George E.},
   author={Askey, Richard},
   author={Roy, Ranjan},
   title={Special functions},
   series={Encyclopedia of Mathematics and its Applications},
   volume={71},
   publisher={Cambridge University Press, Cambridge},
   date={1999},
   pages={xvi+664},
   isbn={0-521-62321-9},
   isbn={0-521-78988-5},
   review={\MR{1688958}},
   doi={10.1017/CBO9781107325937},
}


\bib{ApoLer}{article}{
   author={Apostol, T. M.},
   title={On the Lerch zeta function},
   journal={Pacific J. Math.},
   volume={1},
   date={1951},
   pages={161--167},
   issn={0030-8730},
   review={\MR{43843}},
}


\bib{AR}{article}{
   author={Askey, R. A.},
   author={Roy, R.},
   title={Gamma function},
   conference={NIST handbook of mathematical functions},
   book={
      publisher={U.S. Dept. Commerce, Washington, DC},
   },
   date={2010},
   pages={135--147},
   review={\MR{2655345}},
}

\bib{BB}{article}{
   author={Bailey, David H.},
   author={Broadhurst, David J.},
   title={Parallel integer relation detection: techniques and applications},
   journal={Math. Comp.},
   volume={70},
   date={2001},
   number={236},
   pages={1719--1736},
   issn={0025-5718},
   review={\MR{1836930}},
   doi={10.1090/S0025-5718-00-01278-3},
}


\bib{BadOpe}{article}{
   author={Bade, William G.},
   title={An operational calculus for operators with spectrum in a strip},
   journal={Pacific J. Math.},
   volume={3},
   date={1953},
   pages={257--290},
   issn={0030-8730},
   review={\MR{0055579}},
}


\bib{Bar}{article}{
   author={Barnes, E. W.},
   title={On the theory of the multiple Gamma function},
   journal={Trans. Cambridge Phil. Soc.},
   volume={19},
   date={1904},
   pages={374--425},
}


\bib{BK}{book}{
   author={Berndt, Bruce C.},
   author={Knopp, Marvin I.},
   title={Hecke's theory of modular forms and Dirichlet series},
   series={Monographs in Number Theory},
   volume={5},
   publisher={World Scientific Publishing Co. Pte. Ltd., Hackensack, NJ},
   date={2008},
   pages={xii+137},
   isbn={978-981-270-635-5},
   isbn={981-270-635-6},
   review={\MR{2387477}},
}

\bib{Beu}{article}{
   author={Beurling, Arne},
   title={An extremal property of the Riemann zeta-function},
   journal={Ark. Mat.},
   volume={1},
   date={1951},
   pages={295--300},
   issn={0004-2080},
   review={\MR{40428}},
   doi={10.1007/BF02591365},
}

\bib{BlaThr}{article}{
   author={Blagouchine, Iaroslav V.},
   title={Three notes on Ser's and Hasse's representations for the
   zeta-functions},
   journal={Integers},
   volume={18A},
   date={2018},
   pages={Paper No. A3, 45},
   review={\MR{3777525}},
}



\bib{BBbook}{book}{
   author={Borwein, Jonathan M.},
   author={Borwein, Peter B.},
   title={Pi and the AGM},
   series={Canadian Mathematical Society Series of Monographs and Advanced
   Texts},
   note={A study in analytic number theory and computational complexity;
   A Wiley-Interscience Publication},
   publisher={John Wiley \& Sons, Inc., New York},
   date={1987},
   pages={xvi+414},
   isbn={0-471-83138-7},
   review={\MR{877728}},
}


\bib{BBK}{article}{
   author={Borwein, Jonathan Michael},
   author={Broadhurst, David J.},
   author={Kamnitzer, Joel},
   title={Central binomial sums, multiple Clausen values, and zeta values},
   journal={Experiment. Math.},
   volume={10},
   date={2001},
   number={1},
   pages={25--34},
   issn={1058-6458},
   review={\MR{1821569}},
}


\bib{BoyApo}{article}{
   author={Boyadzhiev, Khristo N.},
   title={Apostol-Bernoulli functions, derivative polynomials and Eulerian
   polynomials},
   journal={Adv. Appl. Discrete Math.},
   volume={1},
   date={2008},
   number={2},
   pages={109--122},
   issn={0974-1658},
   review={\MR{2450260}},
}


\bib{BV}{article}{
   author={Brion, Michel},
   author={Vergne, Mich\`ele},
   title={Lattice points in simple polytopes},
   journal={J. Amer. Math. Soc.},
   volume={10},
   date={1997},
   number={2},
   pages={371--392},
   issn={0894-0347},
   review={\MR{1415319}},
   doi={10.1090/S0894-0347-97-00229-4},
}


\bib{CI}{article}{
   author={Caginalp, Gunduz},
   author={Ion, Bogdan},
   title={Renormalization and analytic continuation},
   journal={},
   volume={},
   date={2020},
   number={},
   pages={},
   status={in preparation},
}


\bib{Carl}{article}{
     author={Carlson, Fritz},
     title={\" Uber ganzwertige Funktionen}, 
     journal={Math. Z.}, 
     volume={11},
     date={1921}, 
     number={1-2}, 
     pages={1--23},
}

\bib{Car}{article}{
   author={Cartier, Pierre},
   title={Fonctions polylogarithmes, nombres polyz\^{e}tas et groupes
   pro-unipotents},
   language={French, with French summary},
   note={S\'{e}minaire Bourbaki, Vol. 2000/2001},
   journal={Ast\'{e}risque},
   number={282},
   date={2002},
   pages={Exp. No. 885, viii, 137--173},
   issn={0303-1179},
   review={\MR{1975178}},
}


\bib{CG}{article}{
   author={Conrey, J. B.},
   author={Ghosh, A.},
   title={On the Selberg class of Dirichlet series: small degrees},
   journal={Duke Math. J.},
   volume={72},
   date={1993},
   number={3},
   pages={673--693},
   issn={0012-7094},
   review={\MR{1253620}},
   doi={10.1215/S0012-7094-93-07225-0},
}


\bib{Dix}{article}{
   author={Dixit, Anup B.},
   title={A uniqueness property of general Dirichlet series},
   journal={J. Number Theory},
   volume={206},
   date={2020},
   pages={123--137},
   issn={0022-314X},
   review={\MR{4013166}},
   doi={10.1016/j.jnt.2019.06.007},
}


\bib{EMOT}{book}{
   author={Erd\'{e}lyi, Arthur},
   author={Magnus, Wilhelm},
   author={Oberhettinger, Fritz},
   author={Tricomi, Francesco G.},
   title={Higher transcendental functions. Vol. I},
   note={Based on notes left by Harry Bateman;
   With a preface by Mina Rees;
   With a foreword by E. C. Watson;
   Reprint of the 1953 original},
   publisher={Robert E. Krieger Publishing Co., Inc., Melbourne, Fla.},
   date={1981},
   pages={xiii+302},
   isbn={0-89874-069-X},
   review={\MR{698779}},
}


\bib{EveMer}{book}{
   author={Everlove, Corey},
   title={Meromorphic Dirichlet Series},
   note={Thesis (Ph.D.)-- The University of Michigan},
   publisher={ProQuest LLC, Ann Arbor, MI},
   date={2018},
   pages={104},
   isbn={},
}


\bib{Eve}{article}{
   author={Everlove, Corey},
   title={Dirichlet series associated to sum-of-digits functions},
   journal={},
   volume={},
   date={2018},
   number={},
   pages={},
   status={arXiv: 1807.07890},
}


\bib{GS}{article}{
   author={Guillera, Jes\'{u}s},
   author={Sondow, Jonathan},
   title={Double integrals and infinite products for some classical
   constants via analytic continuations of Lerch's transcendent},
   journal={Ramanujan J.},
   volume={16},
   date={2008},
   number={3},
   pages={247--270},
   issn={1382-4090},
   review={\MR{2429900}},
   doi={10.1007/s11139-007-9102-0},
}

\bib{Ham1}{article}{
   author={Hamburger, Hans},
   title={\"{U}ber die Riemannsche Funktionalgleichung der $\xi$-Funktion},
   language={German},
   journal={Math. Z.},
   volume={10},
   date={1921}, 
   number={3--4},
   pages={240--254}, 
   issn={0025-5874},
   review={\MR{1544538}},
   doi={10.1007/BF01485292},
}


\bib{Ham2}{article}{
   author={Hamburger, Hans},
   title={\"{U}ber die Riemannsche Funktionalgleichung der $\xi$-Funktion},
   language={German},
   journal={Math. Z.},
   volume={11},
   date={1922}, 
   number={3--4},
   pages={224--245},

}


\bib{Ham3}{article}{
   author={Hamburger, Hans},
   title={\"{U}ber die Riemannsche Funktionalgleichung der $\xi$-Funktion},
   language={German},
   journal={Math. Z.},
   volume={13},
   date={1922},
   number={1},
   pages={283--311},
   issn={0025-5874},
   review={\MR{1544538}},
   doi={10.1007/BF01485292},
}

\bib{HasEin}{article}{
   author={Hasse, Helmut},
   title={Ein Summierungsverfahren f\"ur die Riemannsche $\zeta$-Reihe},
   language={German},
   journal={Math. Z.},
   volume={32},
   date={1930},
   number={1},
   pages={458--464},
   issn={0025-5874},
   review={\MR{1545177}},
}


\bib{HecUbe}{article}{
   author={Hecke, E.},
   title={\"{U}ber die Bestimmung Dirichletscher Reihen durch ihre
   Funktionalgleichung},
   language={German},
   journal={Math. Ann.},
   volume={112},
   date={1936},
   number={1},
   pages={664--699},
   issn={0025-5831},
   review={\MR{1513069}},
   doi={10.1007/BF01565437},
}


\bib{HecEP}{article}{
   author={Hecke, E.},
   title={\"{U}ber Modulfunktionen und die Dirichletschen Reihen mit Eulerscher
   Produktentwicklung. I},
   language={German},
   journal={Math. Ann.},
   volume={114},
   date={1937},
   number={1},
   pages={1--28},
   issn={0025-5831},
   review={\MR{1513122}},
   doi={10.1007/BF01594160},
}


\bib{Hec}{book}{
   author={Hecke, Erich},
   title={Lectures on the theory of algebraic numbers},
   series={Graduate Texts in Mathematics},
   volume={77},
   note={Translated from the German by George U. Brauer, Jay R. Goldman and
   R. Kotzen},
   publisher={Springer-Verlag, New York-Berlin},
   date={1981},
   pages={xii+239},
   isbn={0-387-90595-2},
   review={\MR{638719}},
}


\bib{HPFun}{book}{
   author={Hille, Einar},
   author={Phillips, Ralph S.},
   title={Functional analysis and semi-groups},
   note={Third printing of the revised edition of 1957;
   American Mathematical Society Colloquium Publications, Vol. XXXI},
   publisher={American Mathematical Society, Providence, R. I.},
   date={1974},
   pages={xii+808},
   review={\MR{0423094}},
}


\bib{HirTop}{book}{
   author={Hirzebruch, Friedrich},
   title={Topological methods in algebraic geometry},
   series={Classics in Mathematics},
   note={Translated from the German and Appendix One by R. L. E.
   Schwarzenberger;
   With a preface to the third English edition by the author and
   Schwarzenberger;
   Appendix Two by A. Borel;
   Reprint of the 1978 edition},
   publisher={Springer-Verlag, Berlin},
   date={1995},
   pages={xii+234},
   isbn={3-540-58663-6},
   review={\MR{1335917}},
}

\bib{KanPol}{article}{
   author={Kaneko, Masanobu},
   title={Poly-Bernoulli numbers},
   language={English, with English and French summaries},
   journal={J. Th\'{e}or. Nombres Bordeaux},
   volume={9},
   date={1997},
   number={1},
   pages={221--228},
   issn={1246-7405},
   review={\MR{1469669}},
}


\bib{LL}{article}{
   author={Lagarias, Jeffrey C.},
   author={Li, Wen-Ching Winnie},
   title={The Lerch zeta function II. Analytic continuation},
   journal={Forum Math.},
   volume={24},
   date={2012},
   number={1},
   pages={49--84},
   issn={0933-7741},
   review={\MR{2879971}},
   doi={10.1515/form.2011.048},
}

\bib{LL3}{article}{
   author={Lagarias, Jeffrey C.},
   author={Li, Wen-Ching Winnie},
   title={The Lerch zeta function III. Polylogarithms and special values},
   journal={Res. Math. Sci.},
   volume={3},
   date={2016},
   pages={Paper No. 2, 54},
   issn={2522-0144},
   review={\MR{3465529}},
   doi={10.1186/s40687-015-0049-2},
}


\bib{Kac}{article}{
   author={Kaczorowski, Jerzy},
   title={Axiomatic theory of $L$-functions: the Selberg class},
   conference={Analytic number theory},
   book={
      series={Lecture Notes in Math.},
      volume={1891},
      publisher={Springer, Berlin},
   },
   date={2006},
   pages={133--209},
   review={\MR{2277660}},
   doi={10.1007/978-3-540-36364-4_4},
}


\bib{KMP}{article}{
   author={Kaczorowski, Jerzy},
   author={Molteni, Giuseppe},
   author={Perelli, Alberto},
   title={A converse theorem for Dirichlet $L$-functions},
   journal={Comment. Math. Helv.},
   volume={85},
   date={2010},
   number={2},
   pages={463--483},
   issn={0010-2571},
   review={\MR{2595186}},
   doi={10.4171/CMH/202},
}


\bib{LW}{article}{
   author={Landau, E.},
   author={Walfisz, A.},
   title={\" Uber die {N}ichtfortsetzbarkeit einiger durch {D}irichletsche {R}eihen definierter {Funktionen}},
   language={German},
   journal={Rendiconti di Palermo},
   volume={44},
   date={1919},
   number={},
   pages={82--86},
}

\bib{Lew}{collection}{
   title={Structural properties of polylogarithms},
   series={Mathematical Surveys and Monographs},
   volume={37},
   editor={Lewin, Leonard},
   publisher={American Mathematical Society, Providence, RI},
   date={1991},
   pages={xviii+412},
   isbn={0-8218-1634-9},
   review={\MR{1148371}},
   doi={10.1090/surv/037},
}


\bib{LerNot}{article}{
   author={Lerch, M.},
   title={Note sur la fonction ${\germ K} \left( {w,x,s} \right) =
   \sum\limits_{k = 0}^\infty {\frac{{e^{2k\pi ix} }}{{\left( {w + k}
   \right)^s }}} $},
   language={French},
   journal={Acta Math.},
   volume={11},
   date={1887},
   number={1-4},
   pages={19--24},
   issn={0001-5962},
   review={\MR{1554747}},
   doi={10.1007/BF02418041},
}


\bib{Man}{article}{
   author={Mandelbrojt, S.},
   title={Sur la D\'{e}finition des Fonctions Analytiques},
   language={French},
   journal={Acta Math.},
   volume={45},
   date={1925},
   number={1},
   pages={129--143},
   issn={0001-5962},
   review={\MR{1555193}},
   doi={10.1007/BF02395469},
}



\bib{Mil}{article}{
   author={Milnor, John},
   title={On polylogarithms, Hurwitz zeta functions, and the Kubert
   identities},
   journal={Enseign. Math. (2)},
   volume={29},
   date={1983},
   number={3-4},
   pages={281--322},
   issn={0013-8584},
   review={\MR{719313}},
}


\bib{M-LSur}{article}{
   author={Mittag-Leffler, G.},
   title={Sur la repr\'{e}sentation analytique des fonctions monog\`enes
   uniformes d'une variable ind\'{e}pendante},
   language={French},
   journal={Acta Math.},
   volume={4},
   date={1884},
   number={1},
   pages={1--79},
   issn={0001-5962},
   review={\MR{1554629}},
   doi={10.1007/BF02418410},
}


\bib{Mor}{article}{
   author={Morita, Yasuo},
   title={On the Hurwitz-Lerch $L$-functions},
   journal={J. Fac. Sci. Univ. Tokyo Sect. IA Math.},
   volume={24},
   date={1977},
   number={1},
   pages={29--43},
   issn={0040-8980},
   review={\MR{441924}},
}


\bib{MS}{article}{
   author={Murty, M. Ram},
   author={Sinha, Kaneenika},
   title={Multiple Hurwitz zeta functions},
   conference={Multiple Dirichlet series, automorphic forms, and analytic
      number theory},
   book={
      series={Proc. Sympos. Pure Math.},
      volume={75},
      publisher={Amer. Math. Soc., Providence, RI},
   },
   date={2006},
   pages={135--156},
   review={\MR{2279934}},
   doi={10.1090/pspum/075/2279934},
}


\bib{Oes}{article}{
   author={Oesterl\'{e}, Joseph},
   title={Polylogarithmes},
   language={French, with French summary},
   note={S\'{e}minaire Bourbaki, Vol. 1992/93},
   journal={Ast\'{e}risque},
   number={216},
   date={1993},
   pages={Exp. No. 762, 3, 49--67},
   issn={0303-1179},
   review={\MR{1246392}},
}



\bib{ParInc}{article}{
   author={Paris, R. B.},
   title={Incomplete gamma and related functions},
   conference={NIST handbook of mathematical functions},
   book={
      publisher={U.S. Dept. Commerce, Washington, DC},
   },
   date={2010},
   pages={175--192},
   review={\MR{2655348}},
}


\bib{Per}{article}{
   author={Perelli, Alberto},
   title={General $L$-functions},
   journal={Ann. Mat. Pura Appl. (4)},
   volume={130},
   date={1982},
   pages={287--306},
   issn={0003-4622},
   review={\MR{663975}},
   doi={10.1007/BF01761499},
}


\bib{Pol}{article}{
   author={P\'{o}lya, G.},
   title={\"{U}ber gewisse notwendige Determinantenkriterien f\"{u}r die
   Fortsetzbarkeit einer Potenzreihe},
   language={German},
   journal={Math. Ann.},
   volume={99},
   date={1928},
   number={1},
   pages={687--706},
   issn={0025-5831},
   review={\MR{1512473}},
   doi={10.1007/BF01459120},
}


\bib{KhPu}{article}{
   author={Pukhlikov, A. V.},
   author={Khovanski\u{\i}, A. G.},
   title={The Riemann-Roch theorem for integrals and sums of
   quasipolynomials on virtual polytopes},
   language={Russian, with Russian summary},
   journal={Algebra i Analiz},
   volume={4},
   date={1992},
   number={4},
   pages={188--216},
   issn={0234-0852},
   translation={
      journal={St. Petersburg Math. J.},
      volume={4},
      date={1993},
      number={4},
      pages={789--812},
      issn={1061-0022},
   },
   review={\MR{1190788}},
}


\bib{RO}{article}{
   author={Roy, R.},
   author={Olver, F. W. J.},
   title={Elementary functions},
   conference={NIST handbook of mathematical functions},
   book={
      publisher={U.S. Dept. Commerce, Washington, DC},
   },
   date={2010},
   pages={103--134},
   review={\MR{2655344}},
}


\bib{RuiBar}{article}{
   author={Ruijsenaars, S. N. M.},
   title={On Barnes' multiple zeta and gamma functions},
   journal={Adv. Math.},
   volume={156},
   date={2000},
   number={1},
   pages={107--132},
   issn={0001-8708},
   review={\MR{1800255}},
   doi={10.1006/aima.2000.1946},
}


\bib{Sel}{article}{
   author={Selberg, Atle},
   title={Old and new conjectures and results about a class of Dirichlet
   series},
   conference={Proceedings of the Amalfi Conference on Analytic Number Theory (Maiori,1989)
   },
   book={
      publisher={Univ. Salerno, Salerno},
   },
   date={1992},
   pages={367--385},
   review={\MR{1220477}},
}


\bib{SerSur}{article}{
	author={Ser, Joseph}, 
	title={Sur une expression de la fonction $\zeta(s)$ de Riemann}, 
	journal={C. R. Acad. Sci. Paris (2)}, 
	volume={182},
	date={1926}, 
	pages={1075--1077},
}

\bib{StaEnu}{book}{
   author={Stanley, Richard P.},
   title={Enumerative combinatorics. Vol. 2},
   series={Cambridge Studies in Advanced Mathematics},
   volume={62},
   note={With a foreword by Gian-Carlo Rota and appendix 1 by Sergey Fomin},
   publisher={Cambridge University Press, Cambridge},
   date={1999},
   pages={xii+581},
   isbn={0-521-56069-1},
   isbn={0-521-78987-7},
   review={\MR{1676282}},
   doi={10.1017/CBO9780511609589},
}


\bib{Wei}{article}{
   author={Weil, Andr\'{e}},
   title={\"{U}ber die Bestimmung Dirichletscher Reihen durch
   Funktionalgleichungen},
   language={German},
   journal={Math. Ann.},
   volume={168},
   date={1967},
   pages={149--156},
   issn={0025-5831},
   review={\MR{207658}},
   doi={10.1007/BF01361551},
}


\bib{WidLap}{book}{
   author={Widder, David Vernon},
   title={The Laplace Transform},
   series={Princeton Mathematical Series, v. 6},
   publisher={Princeton University Press, Princeton, N. J.},
   date={1941},
   pages={x+406},
   review={\MR{0005923}},
}


\bib{WorStu}{article}{
   author={Worpitzky, J.},
   title={Studien \"{u}ber die Bernoullischen und Eulerschen Zahlen},
   language={German},
   journal={J. Reine Angew. Math.},
   volume={94},
   date={1883},
   pages={203--232},
   issn={0075-4102},
   review={\MR{1579945}},
   doi={10.1515/crll.1883.94.203},
}


\bib{Zag}{article}{
   author={Zagier, Don},
   title={The dilogarithm function},
   conference={Frontiers in number theory, physics, and geometry. II},
   book={
      publisher={Springer, Berlin},
   },
   date={2007},
   pages={3--65},
   review={\MR{2290758}},
   doi={10.1007/978-3-540-30308-4_1},
}

\end{biblist}
\end{bibdiv}
\end{document}